\documentclass[reqno]{amsart}
\usepackage{amsmath}
\usepackage{amssymb}
\usepackage{hyperref}
\usepackage{soul}
\hypersetup{ colorlinks = true, urlcolor = blue, linkcolor = blue, citecolor = red }
\usepackage{xcolor}
\usepackage{enumitem}
\usepackage{xparse}
\let\realItem\item % save a copy of the original item
\makeatletter
\NewDocumentCommand\myItem{ o }{%
   \IfNoValueTF{#1}%
      {\realItem}% add an item
      {\realItem[#1]\def\@currentlabel{#1}}% add an item and update label
}
\makeatother

\usepackage{enumitem}    
\setlist[enumerate]{
    before=\let\item\myItem,       % use \myItem in enumerate
    label=\textnormal{(\arabic*)}, % format the label
    widest=(2')                    % set the widest label
}
\usepackage{esint}
\makeatletter
\def\namedlabel#1#2{\begingroup
	#2%
	\def\@currentlabel{#2}%
	\phantomsection\label{#1}\endgroup
}
\usepackage{varwidth}
\usepackage{tasks}
\DeclareMathOperator{\dv}{div}
\DeclareMathOperator{\loc}{loc}

\newcommand{\RR}{\mathbb{R}}
\newcommand{\Om}{\Omega}
\newcommand{\na}{\nabla}

\newcommand{\La}{\Lambda}

\newcommand{\ep}{\epsilon}

\newcommand{\la}{\lambda}

\newcommand{\Qla}{Q^\la_\rho}
\newcommand{\Q}{\mathcal{Q}}

\newcommand{\cs}{\textbf{c}}

\theoremstyle{definition}

\newtheorem{theorem}{Theorem}[section]
\newtheorem{definition}[theorem]{Definition} 
\newtheorem{lemma}[theorem]{Lemma}

 \newtheorem{remark}[theorem]{Remark}
\def\Xint#1{\mathchoice
	{\XXint\displaystyle\textstyle{#1}}%
	{\XXint\textstyle\scriptstyle{#1}}%
	{\XXint\scriptstyle\scriptscriptstyle{#1}}%
	{\XXint\scriptstyle\scriptscriptstyle{#1}}%
	\!\int}
\def\XXint#1#2#3{{\setbox0=\hbox{$#1{#2#3}{\int}$}
		\vcenter{\hbox{$#2#3$}}\kern-.5\wd0}}

\def\Yint#1{\mathchoice
	{\YYint\displaystyle\textstyle{#1}}%
	{\YYint\textstyle\scriptstyle{#1}}%
	{\YYint\scriptstyle\scriptscriptstyle{#1}}%
	{\YYint\scriptscriptstyle\scriptscriptstyle{#1}}%
	\!\iint}
\def\YYint#1#2#3{{\setbox0=\hbox{$#1{#2#3}{\iint}$}
		\vcenter{\hbox{$#2#3$}}\kern-.51\wd0}}
\def\longdash{{-}\mkern-3.5mu{-}} 
\def\fiint{\Yint\longdash}

\def\Xint#1{\mathchoice
	{\XXint\displaystyle\textstyle{#1}}%
	{\XXint\textstyle\scriptstyle{#1}}%
	{\XXint\scriptstyle\scriptscriptstyle{#1}}%
	{\XXint\scriptscriptstyle\scriptscriptstyle{#1}}%
	\!\int}
\def\XXint#1#2#3{{\setbox0=\hbox{$#1{#2#3}{\int}$ }
		\vcenter{\hbox{$#2#3$ }}\kern-.6\wd0}}

\def\dashint{\Xint-}

\DeclareMathOperator{\dist}{dist}
\DeclareMathOperator*{\esssup}{ess\,sup}
\usepackage{nameref}
\makeatletter
\let\orgdescriptionlabel\descriptionlabel
\renewcommand*{\descriptionlabel}[1]{%
	\let\orglabel\label
	\let\label\@gobble
	\phantomsection
	\edef\@currentlabel{#1}%
	\let\label\orglabel
	\orgdescriptionlabel{#1}%
}
\makeatother
\numberwithin{equation}{section}
\def\Xint#1{\mathchoice
    {\XXint\displaystyle\textstyle{#1}}%
    {\XXint\textstyle\scriptstyle{#1}}%
    {\XXint\scriptstyle\scriptscriptstyle{#1}}%
    {\XXint\scriptscriptstyle\scriptscriptstyle{#1}}%
    \!\int}
\def\XXint#1#2#3{\setbox0=\hbox{$#1{#2#3}{\int}$}
    \vcenter{\hbox{$#2#3$}}\kern-0.5\wd0}
\def\fint{\Xint-}
\def\dashint{\Xint{\raise4pt\hbox to7pt{\hrulefill}}}

\def\XXiint#1#2#3{\setbox0=\hbox{$#1{#2#3}{\iint}$}
    \vcenter{\hbox{$#2#3$}}\kern-0.5\wd0}

\begin{document}

\title[Parabolic Poincar\'e inequalities and maximal function estimates]{Parabolic Poincar\'e inequalities and maximal function estimates for systems of partial differential equations}

\author{Wontae Kim}
\address[Wontae Kim]{Korea Institute for Advanced Study, 5 Hoegi-ro, Dongdaemun-gu, Seoul 02455, Republic of Korea}
\email{wontae@kias.re.kr}

\author{Juha Kinnunen}
\address[Juha Kinnunen]{Department of Mathematics, Aalto University, P.O. BOX 11100, 00076 Aalto, Finland}
\email[Corresponding author]{juha.k.kinnunen@aalto.fi}

\everymath{\displaystyle}

\makeatletter
\@namedef{subjclassname@2020}{\textup{2020} Mathematics Subject Classification}
\makeatother
\begin{abstract}
We study parabolic Poincar\'e inequalities for solutions to nonlinear systems of partial differential equations. Our main results show that these inequalities are self-improving. As applications, we establish reverse H\"older inequalities for the mean oscillation over parabolic cylinders and for the gradient.
We also consider the corresponding Poincar\'e inequalities and self-improvement results on time intervals at fixed spatial points. These results are based on a pointwise maximal function estimate. As an application, we obtain regularity results in the time direction for solutions to nonlinear systems.
\end{abstract}

\maketitle

\section{Introduction}

We consider the nonlinear system
\begin{equation}\label{p_eq}
    u_t-\dv\mathcal{A}(z, \na u) =-\dv |F|^{p-2}F
\end{equation}
in $\Om_T=\Omega\times(0,T)$, where $1<p<\infty$.
Here $\mathcal{A}(z,\xi):\Omega_T\times \mathbb{R}^{Nn}\to\mathbb{R}^{Nn}$ with $N\ge1$ is a Carath\'eodory vector field, that is, $\mathcal{A}(\cdot,\xi)$ is measurable for fixed $\xi\in\mathbb{R}^{Nn}$ and $\mathcal{A}(z,\cdot)$ is continuous for almost every $z=(x,t)\in\Om_T$. We assume that there exist constants $0<\nu<L<\infty$ such that
\begin{equation}\label{p_ellipticity}
    \nu|\xi|^p\le \mathcal{A}(z,\xi)\cdot \xi\quad\text{and} \quad |\mathcal{A}(z,\xi)|\le L|\xi|^{p-1}
\end{equation}
for almost every $z\in\Om_T$ and for every $\xi\in\mathbb{R}^{Nn}$.

The main challenge in studying \eqref{p_eq} is that an intrinsic scaling arises in estimates when $p\ne2$, see Section \ref{sec_p-growth}. This reflects the fact that the equation lacks the usual homogeneity property: if $u$ is a solution to \eqref{p_eq}, then, in general, $cu$ is not a solution for any constant $c>0$ with $c\neq 1$. Nevertheless, this lack of homogeneity can be compensated for by introducing an intrinsic scaling.
When $p=2$, the system \eqref{p_eq} is invariant under the scaling 
\[
u(x,t)\mapsto u(rx,r^2t)
\quad\text{and}\quad 
F(x,t)\mapsto rF(rx,r^2t),
\]
where $r>0$. Consequently, uniform estimates can be obtained in parabolic cylinders, see Section \ref{sec_2-growth}.
In the general case $1<p<\infty$, the system is invariant under the scaling 
\[
u(x,t)\mapsto (\la r)^{-1} u(rx,\la^{2-p} r^2t)
\quad\text{and}\quad 
F(x,t)\mapsto \la^{-1} F(rx,\la^{2-p} r^2t),
\]
where $\la>0$ is a scaling factor.
The corresponding geometry of the so-called intrinsic cylinders is discussed in \cite{MR1230384}.
The choice $\lambda=r^{-1}$ is useful for H\"older continuity estimates for the solution and recovers the scaling above when $p=2$, see \cite{MR2865434}.
For parabolic Poincar\'e inequalities discussed in Section \ref{sec_p-growth}, the scaling factor naturally depends on the gradient, as in stopping time arguments in gradient regularity results in  \cite{MR2286632,MR1749438,KL_veryweak,MR2968162}.
For this reason, we discuss the cases $p=2$ and $1<p<\infty$ separately. Our results seem to be new already when $p=2$.

The standard Poincar\'e inequality holds for all Sobolev functions, whereas the parabolic Poincar\'e inequality considered here holds for solutions to \eqref{p_eq}.  More precisely, if $u$ is a solution to  \eqref{p_eq} with $1<p<\infty$, there exists a constant $c=c(n,p,L)$ such that
\[
\fiint_{Q}\frac{|u-u_{Q }|^{p}}{\rho^{p}}\,dz 
\le c\fiint_{Q}|\na u|^{p}\,dz
+ c\biggl( \lambda^{2-p}\fiint_{Q} ( |\na u|+ |F| )^{p-1}\,dz \biggr)^p
\]
for every parabolic cylinder $Q=B_\rho(x_0)\times (t_0-\la^{2-p}\rho^2,t_0+\la^{2-p}\rho^2) \Subset\Omega_T$, where $\lambda>0$ is a scaling factor and $B_\rho(x_0)=\{y\in \mathbb{R}^n: |y-x_0|<\rho\}$.
See Lemma \eqref{p_poincare_lem} for the precise assumptions.
When $p=2$, the parabolic Poincar\'e inequality takes the simpler form
\[
\fiint_{Q}\frac{|u-u_{Q }|^{2}}{\rho^{2}}\,dz
\le c\fiint_{Q}|\na u|^{2}\,dz    
+c\left(\fiint_{Q } |F| \,dz\right)^{2}
\]
for every parabolic cylinder $Q=B_\rho(x_0)\times(t_0-\rho^2,t_0+\rho^2) \Subset\Omega_T$, see Lemma \ref{poincare_lem}. 
Throughout the paper, barred integrals and $u_Q$ denote the integral average over $Q$.
One of the novelties of our approach is that the parabolic Poincar\'e inequalities also hold for very weak solutions that do not belong to the natural energy space.
In particular, we do not assume any regularity in the time direction.
Theorem \ref{very_weak} and Theorem \ref{p_very_weak_thm} refine the existing result \cite[Theorem 2.8]{KL_veryweak} by showing that the local $L^2$ assumption imposed there on very weak solutions is redundant.

Poincar\'e inequalities are known to be self-improving under a doubling condition on the underlying measure. 
Below the measure related critical dimension, a Poincar\'e inequality implies a parabolic Sobolev--Poincar\'e inequality with a larger exponent on the left-hand side. At the critical exponent there exists an exponential estimate and above the critical exponent there exists a Morrey type estimate.
Such results hold in a general setting of metric measure spaces, see \cite[Theorem 9.1.15]{HKST2015}. 
The corresponding result for the parabolic Poincar\'e inequality for solutions to \eqref{eq} with $p=2$ follows immediately from this general theory, see Theorem \ref{main_left}.

The general case $1<p<\infty$, with the intrinsic scaling, is more challenging, see Theorem \ref{p_main_left}. The proof involves a scaling argument and two types of intrinsic cylinders. We apply the self-improving properties of the parabolic Poincar\'e inequalities to establish reverse H\"older inequalities for the mean oscillation over parabolic cylinders; see Theorem \ref{main_right_coro} and Theorem \ref{p_rev_u}. We also show how known reverse H\"older inequalities, and the corresponding higher integrability, for the gradient of a solution can be derived from parabolic Poincar\'e inequalities, see Theorem \ref{grad_higher} and Theorem \ref{p_rev_grad}.

We also discuss the corresponding Poincar\'e inequalities over time intervals at a fixed spatial point. This is based on a pointwise maximal function estimate analogous to the Haj\l asz gradient in metric measure spaces. More precisely, if $u$ is a solution to  \eqref{p_eq} with $p=2$, by Lemma \ref{point_wise} there exists a constant $c=c(n,L)$ such that 
\begin{align*}
\frac{|u(x,t) -u(x,s) |}{|t-s|^\frac{1}{2} } 
\le c \bigl(g(x,t)+g(x,s) \bigr),
\end{align*}
for every parabolic cylinder $Q=B\times I=B_\rho(x_0)\times(t_0-\rho^2,t_0+\rho^2) $ such that $2Q=B_{2\rho}(x_0)\times(t_0-(2\rho)^2,t_0+(2\rho)^2)\Subset \Om_T$ and almost every $x\in B$ and $t,s\in I$. 
Here 
\[
g=M^*(|\na u|+|F|) 1_{2Q},
\]
where $M^*f$ is a suitable version of the strong maximal function, see \eqref{strong_mf}.
The corresponding estimate in the general case $1<p<\infty$ is given in Lemma \ref{p_point_wise}.
By integrating the pointwise inequality above, we obtain a Poincar\'e inequality in the time direction, see Theorem \ref{time_poincare} and Theorem \ref{p_time_poincare}.
Self-improving properties of the Poincar\'e inequality in the time direction are discussed in Theorem \ref{time_der} and Theorem \ref{thm_p_time_der}. 
In particular, our results extend the time regularity result in \cite[Theorem 1.2]{MR4242322}. Instead of Fourier analytic techniques, our approach is based on properties of the Poincar\'e inequality in the time direction.

\section{Preliminaries}
Let $\Om\subset\mathbb{R}^n$, $n\ge1$, be a bounded open domain and $T>0$. We denote the space-time cylinder $\Om_T=\Om\times (0,T]$. 
Let $1\le q<\infty$ and $N\ge1$. 
For a measurable function $u=u(z)=u(x,t): \Om_T \to \mathbb{R}^N$, we say that $u\in L^q(0,T;W^{1,q} (\Om;\mathbb{R}^N) )$ 
if $u(\cdot,t)\in W^{1,q}(\Om;\mathbb{R}^N)$ for almost every $0<t<T$ with respect to the one-dimensional Lebesgue measure and
\[
\iint_{\Om_T}(|u|^q+|\na u|^q)\,dz<\infty,
\]
where $dz=dx\,dt$.

\begin{definition}\label{def_weak_sol}
    Let $1<p<\infty$ and assume that $F\in L^p(\Omega_T;\mathbb{R}^{Nn})$. 
    A function $u\in  L^p_{\loc}(0,T;W^{1,p}_{\loc} (\Om;\mathbb{R}^N))$ is a weak solution to \eqref{p_eq} in $\Om_T$ if
\begin{equation}\label{p_weak_sol}
\iint_{\Om_T} (-u\cdot \varphi_t+\mathcal{A}(z,\na u) \cdot \na \varphi) \,dz
=\iint_{\Om_T} |F|^{p-2}F\cdot \na \varphi\,dz
\end{equation}
 for every $\varphi\in C_0^\infty(\Om_T;\mathbb{R}^N)$.
\end{definition}

\begin{remark}\label{rem_sol}
    Integrals in \eqref{p_weak_sol} in the definition of weak solution are well defined under the assumptions $u\in  L^q_{\loc}(0,T;W^{1,q}_{\loc} (\Om;\mathbb{R}^N))$  for some $q$ with $\max\{1,p-1\}\le q<\infty$, and $|F|\in L^{\max\{1,p-1\}}_{\loc}(\Omega_T)$. 
    If $p\le q<\infty$, such solutions are weak solutions, since 
    \[
    L^q_{\loc}(0,T;W^{1,q}_{\loc} (\Om;\mathbb{R}^N))
    \subset L^p_{\loc}(0,T;W^{1,p}_{\loc} (\Om;\mathbb{R}^N)).
    \]
    However, if $\max\{1,p-1\}\le q<p$, then such solutions do not necessarily belong to $L^p_{\loc}(0,T;W^{1,p}_{\loc} (\Om;\mathbb{R}^N))$. These kinds of solutions are called very weak solutions, see \cite{KL_veryweak}. 
\end{remark}

The integral average of $f$ over a cylinder $Q=B\times I\subset \mathbb{R}^{n+1}$ is denoted by
\[
f_Q =\fiint_Q f(\zeta)\,d\zeta.
\]
We also denote 
\[
f_B(t)=\fint_{B}f(x,t)\,dx
\quad\text{and}\quad
f_I(x)=\fint_{I}f(x,t)\,dt
\]
for the integral averages over a fixed time and a fixed space variable, respectively.
For a vector valued function, the integral averages are taken componentwise.

We will apply the following maximal function $M^*f$ of $f:\mathbb R^{n+1}\to\mathbb R^N$, $f\in L^1_{\loc}(\RR^{n+1};\mathbb R^N)$ defined as
 \begin{equation}\label{strong_mf}
 M^*f(z)=\sup_{B\times I\ni z}\fiint_{B\times I}|f(\zeta)|\ d\zeta,
 \end{equation}
 where supremum is taken over all cylinders $B\times I=B_\rho\times(t-s,t+s)\subset \mathbb{R}^{n+1}$, $\rho>0$, $s>0$, such that $z\in B\times I$. Since it is pointwise bounded by taking the supremum over balls and then the supremum over time intervals, there exists a constant $c=c(n,p)$ such that
\[
\iint_{\mathbb R^{n+1}}|M^*f|^p\,dz
\le c\iint_{\mathbb R^{n+1}}|f|^p\,dz,
\]
whenever $1<p<\infty$.

\section{Systems with 2-growth}\label{sec_2-growth}

In this section, we discuss the parabolic Poincar\'e inequality for the system \eqref{p_eq} with $p=2$, that is,
\begin{equation}\label{eq}
    u_t-\dv\mathcal{A}(z, \na u) =-\dv F.
\end{equation} 
   Let $F\in L^2(\Omega_T;\mathbb{R}^{Nn})$. According to Definition \ref{def_weak_sol}, $u\in  L^2_{\loc}(0,T;W^{1,2}_{\loc} (\Om;\mathbb{R}^N))$ is a weak solution to \eqref{eq} in $\Om_T$, if
\begin{equation}\label{weak_sol}
\iint_{\Om_T} (-u\cdot \varphi_t+\mathcal{A}(z,\na u) \cdot \na \varphi) \,dz
=\iint_{\Om_T} F\cdot \na \varphi\,dz
\end{equation}
 for every $\varphi\in C_0^\infty(\Om_T;\mathbb{R}^N)$.

The proper estimates for \eqref{eq} are obtained in a parabolic geometry.
For $z=(x,t)\in\mathbb{R}^{n+1}$ with $x\in\mathbb{R}^n$ and $t\in\mathbb{R}$, the parabolic cylinder of radius $\rho>0$ centered at $z$ is defined as 
\[
Q_{\rho}=Q_{\rho}(z)=B_\rho(x)\times I_\rho(t),
\]
where
\[
B_\rho(x)=\{y\in \mathbb{R}^n: |x-y|<\rho\} 
\quad\text{and}\quad 
I_\rho(t)=(t-\rho^2,t+\rho^2).
\]
For any dilation factor $\alpha>0$, we denote 
\[
\alpha Q_\rho(z)=B_{\alpha\rho}(x)\times I_{\alpha\rho}(t), 
\quad\alpha B_\rho(x)=B_{\alpha\rho}(x)
\]
and
\[
\alpha I_{\rho}(t)=(t-(\alpha\rho)^2,t+(\alpha\rho)^2).
\]
The parabolic metric associated with parabolic cylinders is defined as
\begin{equation}\label{parabolic_metric}
    \dist(z,z')= \max\{ |x-x'|,|t-t'|^\frac{1}{2} \},
\end{equation}
where $z=(x,t)$ and $z'=(x',t')$ for $x,x'\in \mathbb{R}^{n}$ and $t,t'\in \mathbb{R}$.
We note that
\begin{equation}\label{vol_dec}
\frac{|Q_\rho|}{|Q_{\rho'}|}=\left(\frac{\rho}{\rho'}\right)^{n+2},
\end{equation}
where $Q_\rho$ and $Q_{\rho'}$ are parabolic cylinders in $\mathbb R^{n+1}$.
In particular, $|2Q|=2^{n+2}|Q|$ for every parabolic cylinder $Q$ in $\mathbb R^{n+1}$ and thus the $(n+1)$-dimensional Lebesgue measure is doubling with respect to the parabolic metric. Note that the volume decay over parabolic cylinders is of order $n+2$.

The uncentered Hardy-Littlewood maximal function of $f:\mathbb R^{n+1}\to\mathbb R^N$, $f\in L^1_{\loc}(\RR^{n+1},\mathbb R^N)$, with respect to the parabolic metric is defined as
\[
Mf(z)=\sup_{Q\ni z}\fiint_{Q}|f(\zeta)|\ d\zeta,
\]
where supremum is taken over all parabolic cylinders $Q=Q_\rho\subset \mathbb{R}^{n+1}$, $\rho>0$, such that $z\in Q$.
By the maximal function theorem, there exists a constant $c=c(n,p)$ such that
\[
\iint_{\mathbb R^{n+1}}|Mf|^p\,dz
\le c\iint_{\mathbb R^{n+1}}|f|^p\,dz,
\]
whenever $1<p<\infty$. For $p=1$ we have the corresponding weak type bound.

For weak solutions we have the standard Caccioppoli inequality.

\begin{lemma}\label{caccio_lem}
	Let $|F|\in L^2_{\loc}(\Omega_T)$ and assume that $u\in  L^2_{\loc}(0,T;W^{1,2}_{\loc} (\Om;\mathbb{R}^N))$ is a weak solution to \eqref{eq}. Then $u\in L^\infty_{\loc}(0,T;L^2_{\loc}(\Om;\mathbb{R}^N))$. 
    Moreover, let $Q_{2\rho}=Q_{2\rho}(z_0)=B_{2\rho}(x_0)\times I_{2\rho}(t_0)\Subset\Omega_T$.
    Then there exists a constant $c=c(n,N,\nu,L)$ such that, for any $0<\rho<\rho_1<\rho_2< 2\rho$, we have
    \begin{equation}\label{cacc_ie}
	\begin{split}
			&\esssup_{t\in I_{\rho_1}(t_0)}\fint_{B_{\rho_1}(x_0)}\frac{|u-u_{Q_{\rho_1}(z_0)}|^2}{\rho_1^2}\,dx+\fiint_{ Q_{\rho_1}(z_0)}|\na u|^2\,dz\\
			&\qquad\le c\fiint_{Q_{\rho_2}(z_0)} \frac{|u-u_{Q_{\rho_2}(z_0)}|^2}{(\rho_2-\rho_1)^2}\,dz +c\fiint_{Q_{\rho_2}(z_0)}  |F|^2  \,dz
	\end{split}
    \end{equation}
\end{lemma}

Next we introduce the research subject of this paper, the parabolic Poincar\'e inequality of solutions to \eqref{eq}. We state a general version of the result which covers the full range of solutions $u\in  L^q_{\loc}(0,T;W^{1,q}_{\loc} (\Om;\mathbb{R}^N))$, $1\le q<\infty$, see Remark \ref{rem_sol}. In particular, it holds for very weak solutions to \eqref{eq}.

\begin{lemma}\label{poincare_lem}
 Let $1\le q<\infty$. Assume that $|F|\in L^1_{\loc}(\Omega_T)$ and that $u\in  L^q_{\loc}(0,T;W^{1,q}_{\loc} (\Om;\mathbb{R}^N))$ satisfies \eqref{weak_sol}
 for every $\varphi\in C_0^\infty(\Om_T;\mathbb{R}^N)$.
Then there exists a constant $c=c(n,q,L)$ such that
	\begin{equation}\label{poincare_ie}
    \fiint_{Q}\frac{|u-u_{Q }|^{q}}{\rho^{q}}\,dz
            \le c\fiint_{Q}|\na u|^{q}\,dz       
            +c\left(\fiint_{Q } |F| \,dz\right)^{q}
    \end{equation}
    for every parabolic cylinder $Q=Q_{\rho} \Subset\Omega_T$.
\end{lemma}

\begin{proof}
    We denote $Q=B\times I$ and
    \[
    u_{B}(t)=\fint_{B} u(x,t)\,dx,
    \]
    where $0<t<T$.
    By the Poincar\'e inequality in the spatial direction, we conclude that
    \begin{equation}\label{poincare_cal_1}
    \begin{split}
    \fiint_{Q}\frac{|u-u_{Q }|^{q}}{\rho^{q}}\,dz
    &=\fiint_{Q }\frac{|u(x,t)-u_{B}(t)+ u_{B}(t) -u_{Q}|^{q}}{\rho^{q}}\,dz\\
            &\le c\fiint_{Q }\frac{|u(x,t)-u_{B}(t) |^q }{\rho^{q}}\,dz 
            + c\fiint_{Q }\frac{|u_{B}(t) -u_{Q}|^{q}}{\rho^{q}}\,dz\\
            &\le c\fiint_{Q } |\na u|^q\,dz
            + c\fiint_{Q }\frac{|u_{B}(t) -u_{Q}|^{q}}{\rho^{q}}\,dz
    \end{split}
    \end{equation}
    for some constant $c=c(n,q)$.
    
    To estimate the second term on the right-hand side of \eqref{poincare_cal_1}, we observe that
    \[
    u_{Q } =\fint_{I }\fint_{B} u(y,s)\,dy\,ds= \fint_{I} u_{B}(s)\,ds
    \]
    and obtain
   \begin{align*}
       \begin{split}
            \fiint_{Q }|u_{B }(t) -u_{Q}|^q \,dz
            &= \fint_{I} |u_{B}(t) -u_{Q }|^q \,dt  \\
    &=  \fint_{I} \left| \fint_{I}  ( u_{B }(t) -u_{B }(s)   ) \,ds \right|^q \,dt\\
    &\le \fint_{I} \fint_{I} |u_{B}(t)- u_{B}(s)|^q \,ds\,dt.
       \end{split}
   \end{align*}
   To proceed further, we will estimate the difference of integral averages over time slices by using \eqref{eq}. Consider a non-negative cutoff function $\varphi\in C_0^1(B)$ in the spatial direction such that
    \begin{equation}\label{gluing_eta}
        \varphi_{B}=\fint_{B} \varphi\,dx=1
        \quad\text{and}\quad 
        \| \varphi \|_{L^\infty(B )}+ \rho\| \na \varphi \|_{L^\infty( B )} \le c
    \end{equation}
    for a constant $c$. Then we have
    \begin{equation}\label{average_diff_cal}
        \begin{split}
            |u_{B }(t)- u_{B}(s)|^q
            &\le c|u_{B }(t)- (u\varphi )_{B}(t)|^q + c| u_{B}(s)- (u\varphi)_{B}(s)|^q \\
            &\qquad + c| (u\varphi )_{B}(t) - (u\varphi)_{B}(s) |^q 
        \end{split}
    \end{equation}
    for some constant $c=c(q)$. Therefore, we have
    \begin{align*}
    \fiint_{Q }|u_{B }(t) -u_{Q }|^q \,dz
            &\le c\fint_{I }|u_{B }(t) -(u\varphi)_{B }(t)|^q \,dt\\
            &\qquad+ c\esssup_{ t,s\in I } |(u\varphi)_{B }(t)- (u\varphi)_{B }(s)|^q.
    \end{align*}
    Since $u_{B}(t)=u_{B}(t) \varphi_{B}$, we get
    \begin{align}\label{weighted_poincare}
            \begin{split}
            \fint_{I }|u_{B }(t) -(u\varphi)_{B }(t)|^q \,dt
            &=\fint_{I }\left| \fint_{B }  (u(x,t)-u_{B }(t)    )\varphi(x)  \,dx \right|^q \,dt\\
            &\le c\| \varphi \|_{L^\infty (B  )  }\fiint_{Q } |u(x,t)- u_{B }(t)|^q\,dz\\
            &\le c \rho^q\fiint_{Q } |\na u|^q\,dz
            \end{split}
    \end{align}
    for some constant $c=c(n,q)$. To obtain the last inequality, we used the Poincar\'e inequality in the spatial direction. 
   By \eqref{poincare_cal_1}, we have
    \begin{equation}\label{poincare_cal_2}
    \begin{split}
            \fiint_{Q}\frac{|u-u_{Q}|^{q}}{\rho^{q}}\,dz 
            &\le c\fiint_{Q } |\na u|^q\,dz\\
            &\qquad+c\rho^{-q}\esssup_{t_1,t_2\in I}| (u\varphi)_{B}(t_2) -(u\varphi)_{B}(t_1) |^q
    \end{split}
    \end{equation}
    for some constant $c=c(n,q)$.
    
    To estimate the last term on the right-hand side of \eqref{poincare_cal_2}, we consider the following cutoff function in the time direction. Let $t_1,t_2\in I$ with $t_1<t_2$ and let
    \begin{align*}
        \phi_\delta(t)=
    \begin{cases}
        \frac{t-t_1+\delta}{\delta},&\quad  t_1-\delta\le t\le t_1,\\
        1,&\quad t_1\le t\le t_2,\\
        \frac{-t+t_2+\delta}{\delta},&\quad  t_2 \le t\le t_2+\delta,\\
        0,&\quad\text{otherwise},
    \end{cases}
    \end{align*}
    where $\delta>0$ is sufficiently small so that $\phi_\delta\in C_0^{0,1}(I)$. We have $\phi_\delta \varphi\in C_0^{0,1}(Q )$ and  for each unit vector $e_i\in \mathbb{R}^N$, $i=1,\dots,N$, we may apply $e_i \phi_\delta \varphi  \in C_0^{0,1}(Q ;\mathbb{R}^N)$
    as a test function in \eqref{eq}. It follows that
    \[
    -\iint_{Q}  u\cdot e_i\varphi (\phi_\delta)_t\,dz  = \iint_{Q}  (  -\mathcal{A}(z,\na u) + F )\cdot e_i\na \varphi \phi_\delta  \,dz,
    \]
     $i=1,\dots,N$, and thus we may write it as a vector valued integral as
    \[
    -\iint_{Q}  u\varphi (\phi_\delta)_t\,dz  = \iint_{Q}  (  -\mathcal{A}(z,\na u) + F )\cdot \na \varphi \phi_\delta  \,dz.
    \]
    For the left-hand side, we have
    \[
    -\iint_{Q}  u\varphi (\phi_\delta)_t\,dz
    = \frac{1}{\delta}\int_{t_2}^{t_2+\delta} \int_{B} u\varphi \,dx\,dt - \frac{1}{\delta}\int_{t_1-\delta}^{t_1} \int_{B} u\varphi \,dx\,dt
    \]
    while, on the right-hand side, we use \eqref{p_ellipticity} and \eqref{gluing_eta} to get
    \[
    \left| \iint_{Q }  (  -\mathcal{A}(z,\na u) + F )\cdot \na \varphi \phi_\delta  \,dz\right| 
    \le \frac{c}{\rho}\iint_{Q }  (L|\na u| +|F|) \,dz
    \]
    for some constant $c$.
    Combining the estimates above, we obtain 
    \[
    \biggl|\frac{1}{\delta}\int_{t_2}^{t_2+\delta} \int_{B} u\varphi \,dx\,dt - \frac{1}{\delta}\int_{t_1-\delta}^{t_1} \int_{B} u\varphi \,dx\,dt \biggr|\\ 
    \le\frac{c}{\rho}\iint_{Q} ( |\na u|+ |F| )\,dz.
    \]
   This implies that  
    \begin{align*}
    &\biggl| \frac{1}{\delta}\int_{t_2}^{t_2+\delta} (u\varphi)_{B }(t)\,dt-\frac{1}{\delta}\int_{t_1}^{t_1-\delta} (u\varphi)_{B }(t)\,dt    \biggr| \\
    &\qquad=\frac{1}{|B|}\biggl|\frac{1}{\delta}\int_{t_2}^{t_2+\delta} \int_{B} u\varphi \,dx\,dt - \frac{1}{\delta}\int_{t_1-\delta}^{t_1} \int_{B} u\varphi \,dx\,dt \biggr|\\ 
     &\qquad\le\frac{c}{\rho|B|}\iint_{Q} ( |\na u|+ |F| )\,dz
    = c\rho\fiint_{Q} ( |\na u|+ |F| )\,dz,
    \end{align*}
    where $c=c(L)$. Recalling that $t_1,t_2$ are chosen to be arbitrary, letting $\delta$ to $0$, we conclude
    \begin{equation}\label{poincare_cal_3}
    \esssup_{t_1,t_2\in I}| (u\varphi)_{B}(t_2) -(u\varphi)_{B}(t_1) |\le c \rho\fiint_{Q} ( |\na u|+ |F| )\,dz.
    \end{equation}
    Hence, by \eqref{poincare_cal_2} we arrive at
    \begin{align*}
    \fiint_{Q}\frac{|u-u_{Q }|^{q}}{\rho^{q}}\,dz
    &\le c\fiint_{Q } |\na u|^q \,dz\\
    &\qquad+c\rho^{-q}\esssup_{t_1,t_2\in I}| (u\varphi)_{B}(t_2) -(u\varphi)_{B}(t_1) |^q\\
    &\le c\fiint_{Q } |\na u|^q \,dz
    +c \left(\fiint_{Q} ( |\na u|+ |F| )\,dz\right)^q\\
    &\le c\fiint_{Q } |\na u|^q \,dz+c\left( \fiint_{Q} |F| \,dz \right)^q.
    \end{align*}
    This completes the proof.
\end{proof}

\section{Self-improving properties of systems with 2-growth}
In this section we discuss a self-improving property of the parabolic Poincar\'e inequality. Assume that $1\le q<\infty$ and $|F|\in L^q_{\loc}(\Om_T)$. Then \eqref{poincare_ie} implies that
   \begin{equation}\label{poincare_iep}
 \fiint_{Q}\frac{|u-u_{Q }|^{q}}{\rho^{q}}\,dz
            \le c\fiint_{Q}(|\na u|+|F|)^q \,dz
\end{equation}
for every parabolic cylinder $Q=Q_{\rho} \Subset\Omega_T$.
 For the proof, we refer to the corresponding proof for a Poincar\'e inequality on metric measure spaces with the volume decay bound \eqref{vol_dec}, see \cite[Theorem 9.1.15]{HKST2015}.
We note that the proof only applies the parabolic Poincar\'e inequality \eqref{poincare_iep} and does not apply any other information on the system \eqref{eq}.

\begin{theorem}\label{main_left}
Let $1\le q<\infty$. Assume that $|F|\in L^q_{\loc}(\Omega_T)$ and that $u\in  L^q_{\loc}(0,T;W^{1,q}_{\loc} (\Om;\mathbb{R}^N))$ satisfies \eqref{weak_sol} for every $\varphi\in C_0^\infty(\Om_T;\mathbb{R}^N)$. 
\begin{itemize}
    \item[(i)] If $1\le q<n+2$, then for $1\le q\le q^*<\tfrac{q(n+2)}{n+2-q}$, there exists a constant $c=c(n,q,q^*,L)$ such that 
    \[
    \left( \fiint_{Q } \frac{|u-u_{Q}|^{q^*}}{\rho^{q^*}}\,dz\right)^\frac{1}{q^*}
    \le c\left( \fiint_{Q } ( |\na u| + |F| ) ^q\,dz\right)^\frac{1}{q}
    \]
    for every parabolic cylinder $Q=Q_{\rho}\Subset\Omega_T$.
    \item[(ii)] If $q=n+2$, then there exists a constant $c=c(n,L)$ such that
    \[
    \fiint_Q\exp\left(\left(
   \frac{|u-u_{Q}|}{\displaystyle c\rho\left(\fiint_{Q}(|\na u|+|F|)^{n+2} \,dz\right)^{\frac{1}{n+2}}}\right)^{\frac{n+2}{n+1}}\right)\,dz\le c
    \]
     for every parabolic cylinder $Q=Q_{\rho}\Subset\Omega_T$.
    \item[(iii)] If $q>n+2$, then there exists a constant $c=c(n,q,L)$ such that
    \[
    \esssup_Q\frac{|u-u_{Q}|}{\rho}
    \le c\left( \fiint_{Q } ( |\na u| + |F| ) ^q\,dz\right)^\frac{1}{q}
    \]
     for every parabolic cylinder $Q=Q_{\rho}\Subset\Omega_T$.
\end{itemize}
\end{theorem}

\begin{remark}\label{main_left_rmk}
Let $1\le q<\infty$. Assume that $|F|\in L^q_{\loc}(\Omega_T)$ and that $u\in  L^q_{\loc}(0,T;W^{1,q}_{\loc} (\Om;\mathbb{R}^N))$ satisfies \eqref{weak_sol} for every $\varphi\in C_0^\infty(\Om_T;\mathbb{R}^N)$.
\begin{itemize}
\item[(1)] If $1\le q<n+2$, then $u\in  L^q_{\loc}(\Om_T;\mathbb{R}^N)$ for every $q$ with  $1\le q\le q^*<\tfrac{q(n+2)}{n+2-q}$.
\item[(2)] If $q=n+2$, then
\begin{align*}
    \fiint_{Q}|u-u_{Q }|\,dz
    &\le\left(\fiint_{Q}|u-u_{Q }|^{n+2}\,dz\right)^{\frac{1}{n+2}}\\
    &\le c\rho\left(\fiint_{Q}(|\na u|+|F|)^{n+2}\,dz\right)^{\frac{1}{n+2}}\\
    &\le c\left(\iint_{Q}(|\na u|+|F|)^{n+2}\,dz\right)^{\frac{1}{n+2}}.
\end{align*}
It follows that $u$ is a function of bounded mean oscillation (BMO) over parabolic cylinders that are contained in every compact subset of $\Om_T$.
\item[(3)] If $q>n+2$, then
\begin{align*}
    |u(z)-u(z')|
    &\le|u(z)-u_Q|+|u_Q-u(z')|\\
    &\le\esssup_Q|u-u_{Q}|\\
    &\le c\rho\left( \fiint_{Q } ( |\na u| + |F| ) ^q\,dz\right)^\frac{1}{q}\\
    &\le c\rho^{1-\frac{n+2}{q}}\left( \iint_{Q } ( |\na u| + |F| ) ^q\,dz\right)^\frac{1}{q}
\end{align*}
for almost every $z,z'\in Q$. It follows that $u$ is H\"older continuous in every compact subset of $\Om_T$.
\end{itemize}
\end{remark}

Next we present applications of the results above for solutions to \eqref{eq}.
We begin with a reverse H\"older inequality for the mean oscillation over parabolic cylinders.

\begin{theorem}\label{main_right_coro}
Let $|F|\in L^2_{\loc}(\Om_T)$and assume that $u\in L^2_{\loc}(0,T;W^{1,2}_{\loc}(\Om;\mathbb{R}^N))$ is a weak solution to \eqref{eq}. 
    Let $2\le q<2(1+\tfrac{2}{n})$ and let $Q =Q_\rho$ be a  parabolic cylinder such that $2Q  \Subset \Om_T$.
    Then there exists a constant $c=c(n,N,q,\nu,L)$ such that
    \begin{equation}\label{rh_rhs}
    \biggl(\fiint_{Q  } \frac{|u-u_{Q  }|^q}{\rho^q}\,dz\biggr)^\frac{1}{q}
            \le  c\fiint_{ 2Q  } \frac{|u-u_{2Q }| }{ 2\rho}\,dz\\
            + c\biggl(\fiint_{2Q  }  |F|^2  \,dz\biggr)^\frac{1}{2}.
    \end{equation}
       \end{theorem}

\begin{proof}
 Let $0< \rho < \rho_1< \rho_2<2\rho$ and denote $Q_1=\tfrac{\rho_1}{\rho}Q $ and $Q_2=\tfrac{\rho_2}{\rho}Q $. 
    By Theorem~\ref{main_left}, for any $2\le q<2(1+\tfrac{2}{n})$, there exists a constant $c=c(n,q,L)$ such that
    \begin{align*}
            &\biggl(\fiint_{Q_1} \frac{|u-u_{Q_1 }|^q}{\rho_1^q}\,dz\biggr)^\frac{1}{q}
            \le  c\biggl(\fiint_{ Q_1} (|\na u| + |F|)^2\,dz \biggr)^\frac{1}{2}\\
            &\qquad\le  c\biggl(\fiint_{Q_1}  |\na u|^2  \,dz\biggr)^\frac{1}{2}
             +c\biggl(\fiint_{Q_1}  |F|^2  \,dz\biggr)^\frac{1}{2}.
    \end{align*}
    For the first term on the right-hand side, by Lemma~\ref{caccio_lem}, we have
      \begin{align*}
            \biggl(\fiint_{Q_1}  |\na u|^2  \,dz\biggr)^\frac{1}{2}
            \le  c\biggl(\fiint_{Q_2} \frac{|u-u_{Q_2}|^2}{(\rho_2- \rho_1)^2}\,dz \biggr)^\frac{1}{2}
             +c\biggl(\fiint_{Q_2}  |F|^2  \,dz\biggr)^\frac{1}{2}
    \end{align*}
    for some constant $c=c(n,N,\nu,L)$.
    It follows that
      \begin{align*}
            \biggl(\fiint_{Q_1} \frac{|u-u_{Q_1 }|^q}{\rho_1^q}\,dz\biggr)^\frac{1}{q}
            &\le  c\biggl(\fiint_{Q_2} \frac{|u-u_{Q_2}|^2}{(\rho_2- \rho_1)^2}\,dz \biggr)^\frac{1}{2}
             +c\biggl(\fiint_{Q_2}  |F|^2  \,dz\biggr)^\frac{1}{2}
    \end{align*}
     for some constant $c=c(n,N,q,\nu,L)$.
    We estimate the first term on the right-hand side as
    \begin{align*}
            &\biggl(\fiint_{Q_2} \frac{|u-u_{Q_2}|^2}{(\rho_2-\rho_1)^2}\,dz\biggr)^\frac{1}{2}
            =\frac{\rho_2}{\rho_2-\rho_1}\biggl( \fiint_{Q_2} \frac{|u-u_{Q_2}|^2}{\rho_2^2}\,dz \biggr)^\frac{1}{2}\\
            &\qquad\le\frac{\rho_2}{\rho_2-\rho_1}\biggl( \fiint_{Q_2} \frac{|u-u_{Q_2}|^q}{ \rho_2^q}\,dz \biggr)^\frac{\theta}{q}
            \biggl( \fiint_{Q_2} \frac{|u-u_{Q_2}| }{\rho_2}\,dz \biggr)^{1-\theta},
            \end{align*}
    where $0<\theta<1$ is such that $\tfrac{1}{2}=\tfrac{\theta}{q}+1-\theta$. 
    By applying Young's inequality, we obtain
    \begin{align*}
            &\biggl(\fiint_{Q_1} \frac{|u-u_{Q_1}|^q}{\rho_1^q}\,dz\biggr)^\frac{1}{q}
            \le \frac{1}{2}\biggl( \fiint_{Q_2} \frac{|u-u_{Q_2}|^q}{ \rho_2^q}\,dz \biggr)^\frac{1}{q}\\
            &\qquad + c\biggl(\frac{\rho_2}{\rho_2-\rho_1}\biggr)^\frac{1}{1-\theta} \biggl(\fiint_{Q_2} \frac{|u-u_{Q_2}| }{\rho_2}\,dz \biggr)^\frac{1}{2}
             + c\biggl(\fiint_{Q_2}  |F|^2  \,dz\biggr)^\frac{1}{2},
    \end{align*}
    where $c=(n,N,q,\nu,L)$.
    The estimate \eqref{rh_rhs} follows from a standard absorbtion result, see \cite[Lemma~6.1]{MR1962933}.
\end{proof}

The following lemma provides a proof of a reverse H\"older inequality for the gradient, see \cite{MR652852}, by combining the Caccioppoli estimate and parabolic Poincar\'e inequality. 

\begin{theorem}\label{grad_higher}
 Let $|F|\in L^2_{\loc}(\Om_T)$and assume that $u\in L^2_{\loc}(0,T;W^{1,2}_{\loc}(\Om;\mathbb{R}^N))$ is a weak solution to \eqref{eq}.
Then there exist an exponent $q=q(n,L)$, $1<q<2$, and a constant $c=c(n,N,\nu,L)$ such that 
    \begin{equation}\label{rhi_1}
    \biggl(\fiint_{Q }|\na u|^2\,dz\biggr)^\frac{1}{2}
    \le c\biggl( \fiint_{2Q } |\na u| ^q\,dz\biggr)^\frac{1}{q} 
    + c\biggl(\fiint_{2Q } |F|^2 \,dz\biggr)^\frac{1}{2}
    \end{equation}
    for every parabolic cylinder $2Q \Subset \Om_T$.
    Moreover, there exist $c=c(n,N,\nu,L)$ and $\varepsilon_0=\varepsilon_0(n,N,\nu,L)$, $0<\varepsilon_0<1$, such that for every $0<\varepsilon<\varepsilon_0$, we have
     \begin{equation}\label{rhi_2}
    \biggl(\fiint_{Q } |\na u|^{2+\varepsilon}\,dz\biggr)^\frac{1}{2+\epsilon}
    \le c \biggl( \fiint_{2Q } |\na u|^{2}\,dz\biggr)^\frac{1}{2}
    + \biggl(\fiint_{2Q } |F|^{2+\varepsilon}\,dz\biggr)^\frac{1}{2+\varepsilon}
    \end{equation}
    for every parabolic cylinder $2Q \Subset \Om_T$, under the assumption that $|F|\in L^{2+\varepsilon}(2  Q )$.
\end{theorem}
\begin{proof}
Let $Q=Q_\rho$ parabolic cylinder $2Q=Q_{2\rho} \Subset \Om_T$.
 By Theorem~\ref{main_left} there exists $q=q(n)$, $1<q<2$, such that
    \begin{align*}
    \fiint_{ 2Q  } \frac{|u-u_{2 Q }|^2}{(2\rho)^2} \,dz 
    &\le c  \biggl( \fiint_{  2Q } ( |\na u| + |F| ) ^q\,dz\biggr)^\frac{2}{q}\\
    &\le c  \biggl( \fiint_{ 2Q }|\na u|^q\,dz\biggr)^\frac{2}{q} 
    +  c\fiint_{ 2Q }|F|^2\,dz .
    \end{align*}
    for some constant $c=c(n,L)$. 
    Substituting this into the left-hand side of \eqref{cacc_ie}, we have
    \begin{align*}
    \fiint_{ Q }|\na u|^2\,dz  
    &\le  c \fiint_{ 2Q }  \frac{|u-u_{ 2Q } |^2}{(2\rho)^2} \,dz 
    + c  \fiint_{ 2Q }  |F|^2  \,dz \\
    &\le c\biggl( \fiint_{ 2Q } |\na u| ^q\,dz\biggr)^\frac{2}{q} 
    + c \fiint_{ 2Q }  |F|^2  \,dz .
    \end{align*}
    This proves \eqref{rhi_1}. 
    The estimate \eqref{rhi_2} follows from an application of self-improving properties of reverse H\"older inequalities, see \cite{MR652852}. 
\end{proof}

Next, we state a refinement of \cite[Theorem 2.8]{KL_veryweak} for very weak solutions in the case of 2-growth.

\begin{theorem}\label{very_weak}
       There exists $\beta_0=\beta_0(n,N,\nu,L)$, $0<\beta_0<1$, such that if 
        $u\in L_{\loc}^{2-\beta}(0,T;W_{\loc}^{1,2-\beta}(\Om;\mathbb{R}^N))$, 
       $0<\beta<\beta_0$,
       is a very weak solution to \eqref{eq} in $\Om_T$ and $|F|\in L^2_{\loc}(\Om_T)$, then $u\in L_{\loc}^2(0,T;W_{\loc}^{1,2}(\Omega;\mathbb{R}^N))$
       and $u$ is a weak solution to \eqref{eq} in $\Omega_T$.
\end{theorem}

\begin{proof}
    By Theorem~\ref{main_left} (i) we have $u \in L^2_{\loc}(\Omega_T; \mathbb{R}^N)$ for every very weak solution 
    $u\in L_{\loc}^q(0,T; W_{\loc}^{1,q}(\Omega_T; \mathbb{R}^N))$ with $\tfrac{2n+4}{n+4}< q<2$. 
    By \cite[Theorem 2.8]{KL_veryweak},  there exists $\beta_0=\beta_0(n,N,\nu,L)$, $0<\beta_0<1$, such that for every very weak solution 
    \[
    u\in L^2_{\loc}(\Omega_T;\mathbb R^N)\cap L_{\loc}^{2-\beta}(0,T; W_{\loc}^{1,2-\beta}(\Omega_T; \mathbb{R}^N)),
    \quad 0<\beta<\beta_0,
    \]
    we may conclude that $u\in L_{\loc}^2(0,T; W_{\loc}^{1,2}(\Omega_T; \mathbb{R}^N))$. This implies that $u$ is a weak solution to \eqref{eq} in $\Omega_T$.
\end{proof}

\section{Time direction estimates for systems with 2-growth}

Next, we discuss the Hajłasz gradient in the time direction with respect to the parabolic metric \eqref{parabolic_metric}.

\begin{lemma}\label{point_wise}
     Let $|F|\in L^1_{\loc}(\Om_T)$ and assume that $u\in L^1_{\loc}(0,T;W^{1,1}_{\loc}(\Om;\mathbb{R}^N))$ satisfies \eqref{eq} for every $\varphi\in C_0^\infty(\Om_T;\mathbb R^N)$. Then there exists a constant $c=c(n,L)$ such that 
    \[
     \frac{|u(x,t) -u(x,s) |}{|t-s|^\frac{1}{2} } 
     \le c \bigl( M^*g(x,t)+ M^*g(x,s) \bigr),
    \]
    where $g=(|\na u|+|F|)1_{2Q}$, for every parabolic cylinder $Q=Q_\rho=B\times I$ such that $2Q\Subset \Om_T$ and almost every $x\in B$ and $t,s\in I$.
\end{lemma}

\begin{proof}
    Let $x\in B$, $t,s\in I$ and $r =(\tfrac12|t-s|)^{\frac12}$.
Then $B_r(x)\times (t,s)\subset 2Q$ and we have
 \begin{equation}\label{average_cal}
 \begin{split}
        |u(x,t) -u(x,s) | 
            &\le |u(x,t) - u_{B_r(x)}(t) |
            +|u(x,s) -u_{B_r(x)}(s) | \\
             &\qquad+ |u_{B_r(x)}(t)-u_{B_r(x)}(s)|,
            \end{split}
    \end{equation}
    where 
    \[
    u_{B_r(x)}(t)=\fint_{B_r(x)} u(y,t)\,dy.
    \]
To estimate the first two terms on the right-hand side of \eqref{average_cal}, we apply Theorem~\cite[Theorem~3.5]{KLV} to conclude that there exists a constant $c=c(n)$ such that
    \begin{equation}\label{average_cal_2}
    \begin{split}
    |u(x,t)-u_{B_r(x)}(t)|
    &\le cr\sup_{R>0}\fint_{B_R(x)}|\na u(y,t)|1_{B_r(x)}(y,t)\,dy\\
    &\le c r M^*(|\na u|1_{2Q})(x,t)
    \end{split}
    \end{equation}
    and
    \[
    |u(x,s)-u_{B_r(x)}(s)|
    \le c r M^*(|\na u|1_{2Q})(x,s)
    \]
    for almost every $(x,t),(x,s)\in B\times I$.
    
    To estimate the last term on the right-hand side of \eqref{average_cal}, we follow the argument to estimate \eqref{average_diff_cal} in the proof of Lemma~\ref{poincare_lem}. By considering a non-negative function $\varphi\in C_0^1(B_r(x))$ such that \eqref{gluing_eta} holds, we get
     \begin{equation*}
        \begin{split}
            &|u_{B_r(x) }(t)- u_{B_r(x)}(s)|
            \le |u_{B_r(x) }(t)- (u\varphi )_{B_r(x)}(t)|\\
            &\qquad+ | u_{B_r(x)}(s)- (u\varphi)_{B_r(x)}(s)|
             + | (u\varphi )_{B_r(x)}(t) - (u\varphi)_{B_r(x)}(s) |.
        \end{split}
    \end{equation*}
    For the first and the second terms on the right-hand side, we apply the Poincar\'e inequality in the spatial direction as in \eqref{weighted_poincare}. Then, there exists a constant $c=c(n)$ such that
    \begin{equation}\label{average_cal_3}
    \begin{split}
    |u_{B_r(x) }(t)- (u\varphi )_{B_r(x)}(t)|
    &=\left| \fint_{B_r(x)}  (u(y,t)-u_{B_r(x)}(t) )\varphi(y)  \,dy \right|\\
            &\le c\| \varphi \|_{L^\infty (B_r(x))  }\fint_{B_r(x) } |u(y,t)- u_{B_r(x) }(t)|\,dy\\
            &\le c r\fint_{B_r(x) } |\na u(y,t)|\,dy
    \end{split}
    \end{equation}
    and similarly
    \[
    | u_{B_r(x)}(s)- (u\varphi)_{B_r(x)}(s)|
    \le c r\fint_{B_r(x) } |\na u(y,s)|\,dy.
    \]
    For the last term on the right-hand side of \eqref{average_cal} we use \eqref{poincare_cal_3} and conclude that there exists a constant $c=c(n,L)$ such that
    \begin{align*}
        \begin{split}
            &|u_{B_r(x)}(t)-u_{B_r(x)}(s)| 
            \le cr \fiint_{B_r(x)\times (t,s)} (|\na u|+|F|)\,d\zeta\\
            &\qquad\qquad +  cr\fint_{B_r(x)}|\na u(y,t)|\,dy +cr\fint_{B_r(x)}|\na u(y,s)|\,dy \\
            &\qquad\le c r M^*( (|\na u|+|F|)1_{2Q})(x,t)+ c r M^*((|\na u|+|F|)1_{2Q})(x,s)
        \end{split}
    \end{align*}
    for almost every $(x,t),(x,s)\in \Omega\times I$.
    Combining these estimates, there exists a constant $c=c(n,L)$ such that 
    \begin{equation*}
    \begin{split}
     \frac{|u(x,t) -u(x,s) |}{|t-s|^\frac{1}{2} } 
     &\le c \bigl( M^*( (|\na u|+|F|) 1_{2Q})(x,t)\\
     &\qquad+ M^*( (|\na u|+|F|)1_{2Q})(x,s) \bigr)
     \end{split}
    \end{equation*}
    for almost every $(x,t),(x,s)\in B\times I$.
    This completes the proof.
\end{proof}

\begin{theorem}\label{time_poincare}
     Let $1<q<\infty$. Assume that $|F|\in L^q_{\loc}(\Om_T)$ and that $u\in L^q_{\loc}(0,T;W^{1,q}_{\loc}(\Om;\mathbb{R}^N))$ satisfies \eqref{eq} for every $\varphi\in C_0^\infty(\Om_T;\mathbb R^N)$. Then there exists a constant $c=c(n,L)$ such that
    \[
    \fint_I |u(x,t)-u_I(x)|^q \,dt\le  c |I|^\frac{q}{2} \fint_I M^*g(x,t)^q \,dt,
    \]
 where $g=(|\na u|+|F|)1_{2Q}$, for every parabolic cylinder $Q=Q_\rho=B\times I$ such that $2Q\Subset \Om_T$ and almost every $x\in B$.
\end{theorem}

 \begin{proof}
    We apply Lemma~\ref{point_wise} to estimate the left hand side of the desired inequality and obtain
    \begin{align*}
            \fint_I |u(x,t)-u_I(x)|^q \,dt &\le \fint_I\fint_I |u(x,t)-u(x,s)|^q\,ds\,dt\\
            &\le c |I|^\frac{q}{2} \fint_I \fint_I ( M^*( (|\na u|+|F|)1_{2Q})(x,t)\\
            &\qquad+ M^*( (|\na u|+|F|)1_{2Q})(x,s)  )^q \,ds \,dt \\
            &= c |I|^\frac{q}{2} \fint_I  M^*( (|\na u|+|F|)1_{2Q})(x,t)^q\,dt.
    \end{align*}
    Therefore, the proof is completed.
\end{proof}

By \cite[Theorem 9.1.15]{HKST2015}, we obtain self-improving properties of the above Poincar\'e inequality.

\begin{theorem}\label{time_der}
Let $1<q<\infty$. Assume that $|F|\in L^q_{\loc}(\Om_T)$ and that $u\in L^q_{\loc}(0,T;W^{1,q}_{\loc}(\Om;\mathbb{R}^N))$ satisfies \eqref{eq} for every $\varphi\in C_0^\infty(\Om_T;\mathbb R^N)$ and let 
\[
g=M^*(|\na u|+|F|)1_{2Q}.
\]
      \begin{itemize}
    \item[(i)] If $1< q< 2$, then for $1< q \le q^* <\tfrac{2q}{2-q}$, then there exists a constant $c=c(n,q,q^*,L)$ such that 
    \[
    \biggl( \fint_{I } \frac{|u(x,t)-u_{I}(x)|^{q^*}}{|I|^\frac{q^*}{2}}\,dt \biggr)^\frac{1}{q^*}
    \le c\biggl( \fint_{I}g(x,t)^q\,dt\biggr)^\frac{1}{q}
    \]
    for every parabolic cylinder $Q=Q_\rho=B\times I$ such that $2Q\Subset \Om_T$ and for almost every $x\in B$.
    \item[(ii)] If $q=2$, then there exists a constant $c=c(n,L)$ such that
    \[
    \fint_I\exp\left(\left(
   \frac{|u(x,t)-u_{I}(x)|}{\displaystyle c|I|^\frac{1}{2}\left( \fint_{I}  g(x,t)^2 \,dt\right)^{\frac{1}{2}}}\right)^{2}\right)\,dt\le c
    \]
    for every parabolic cylinder $Q=Q_\rho=B\times I$ such that $2Q\Subset \Om_T$ and for almost every $x\in B$.
    \item[(iii)] If $q>2$, then there exists a constant $c=c(n,q,L)$ such that
    \[
    \esssup_{t\in I}\frac{|u(x,t)-u_{I}(x)|}{|I|^\frac{1}{2}}
    \le c\left( \fint_{I}g(x,t)^q\,dt\right)^\frac{1}{q}
    \]
    for every parabolic cylinder $Q=Q_\rho=B\times I$ such that $2Q\Subset \Om_T$ and for almost every $x\in B$.
\end{itemize}
\end{theorem}

\section{Systems with $p$-growth}\label{sec_p-growth}

In this section, we discuss corresponding results for the nonlinear system \eqref{p_eq} with $1<p<\infty$ and consider parabolic cylinders with a scaling factor in the time direction. 

For $z=(x,t)\in\mathbb{R}^{n+1}$ with $x\in\mathbb{R}^n$ and $t\in\mathbb{R}$, the parabolic intrinsic cylinder centered at $z$ of radius $\rho>0$ with a scaling factor $\la>0$ is defined as
\[
\Qla=\Qla(z)=B_\rho(x)\times I^\la_\rho(t),
\]
where
\[
I^\la_\rho(t)=(t-\la^{2-p}\rho^2,t+\la^{2-p}\rho^2).
\]
The related parabolic metric is 
\[
\dist(z,z')= \max\{ |x-x'|,\la^\frac{p-2}{2}|t-t'|^\frac{1}{2} \},
\]
For any dilation factor $\alpha>0$, we denote 
\[
\alpha \Qla (z)=B_{\alpha\rho}(x)\times I^\la_{\alpha\rho}(t)
\]
and
\[
\alpha I^\la_{\rho}(t)=(t-\la^{2-p}(\alpha\rho)^2,t+\la^{2-p}(\alpha\rho)^2).
\]
For fixed $\rho>0$ and $\la>1$, it is customary to use the intrinsic cylinder $Q_\rho^\la$, if $2<p<\infty$, and $Q_{\rho_\la}^\la$ with $\rho_\la=\la^{\frac{p-2}{2}}\rho$, if $1<p<2$. This convention ensures that the intrinsic cylinder is contained within the non-intrinsic one. 

We begin with a parabolic Poincar\'e inequality for solutions to \eqref{p_eq}.

\begin{lemma}\label{p_poincare_lem}
 Let $1< p<\infty$, $\max\{1,p-1\}\le q<\infty$ and $\la>0$. Assume that $|F|\in L^{p-1}_{\loc}(\Omega_T)$ and that $u\in  L^q_{\loc}(0,T;W^{1,q}_{\loc} (\Om;\mathbb{R}^N))$ satisfies \eqref{p_weak_sol}
 for every $\varphi\in C_0^\infty(\Om_T;\mathbb{R}^N)$.
Then there exists a constant $c=c(n,q,L)$ such that
	\begin{equation}\label{p_poincare_ie}
    \fiint_{Q}\frac{|u-u_{Q }|^{q}}{\rho^{q}}\,dz \le c\fiint_{Q}|\na u|^{q}\,dz+ c\biggl( \lambda^{2-p}\fiint_{Q} ( |\na u|+ |F| )^{p-1}\,dz \biggr)^q
    \end{equation}
    for every parabolic cylinder $Q=\Qla \Subset\Omega_T$.
\end{lemma}

\begin{proof}
    The proof is analogous to the proof of Lemma~\ref{poincare_lem}. As in \eqref{poincare_cal_1}, we apply the Poincar\'e inequality in the spatial direction and obtain
    \[
    \fiint_{Q}\frac{|u-u_{Q }|^{q}}{\rho^{q}}\,dz \le c\fiint_{Q}|\na u|^{q}\,dz+ \fiint_{Q}\frac{|u_B(t)-u_{Q }|^{q}}{\rho^{q}}\,dz,
    \]
    where $Q=B_\rho\times I^\la_\rho$, $t\in I^\la_\rho$ and $c=c(n,q)$. To estimate the last term on the right-hand side, we use the cutoff function \eqref{gluing_eta} and apply \eqref{weighted_poincare} to obtain
    \begin{align*}
    \fiint_{Q}\frac{|u-u_{Q }|^{q}}{\rho^{q}}\,dz 
    &\le c\fiint_{Q}|\na u|^{q}\,dz\\
    &\qquad+ c\rho^{-q}\esssup_{t,s\in I^\la} | (u\varphi)_B(t)- (u\varphi)_B(s) |^q.
    \end{align*}
    Taking $\varphi \phi_\delta$ as a test function in \eqref{p_weak_sol}, we get
   \begin{equation}\label{p_poincare_cal}
       \begin{split}
           \esssup_{t_1,t_2\in I^\la}| (u\varphi)_{B}(t_2) -(u\varphi)_{B}(t_1) |
       & \le \frac{c}{\rho|B|} \iint_{Q} ( |\na u|+ |F| )^{p-1} \,dz \\
       &= c \rho\lambda^{2-p}\fiint_{Q} ( |\na u|+ |F| )^{p-1}\,dz.
       \end{split}
   \end{equation}
   Combining the estimates above, we conclude that
   \[
   \fiint_{Q}\frac{|u-u_{Q }|^{q}}{\rho^{q}}\,dz \le c\fiint_{Q}|\na u|^{q}\,dz+ c\biggl( \lambda^{2-p}\fiint_{Q} ( |\na u|+ |F| )^{p-1}\,dz \biggr)^q.
   \]
\end{proof}

Observe that the scaling factor $\lambda^{2-p}$, related to the parabolic cylinders, appears on the right-hand side of the parabolic Poincar\'e inequality \eqref{p_poincare_ie} if $p\ne2$. 

\begin{lemma}
Let $1< p<\infty$, $\max\{1,p-1\}\le q<\infty$ and $\la>0$. 
Assume that $|F|\in L^{q}_{\loc}(\Omega_T)$ and that $u\in  L^q_{\loc}(0,T;W^{1,q}_{\loc} (\Om;\mathbb{R}^N))$ satisfies \eqref{p_weak_sol}
 for every $\varphi\in C_0^\infty(\Om_T;\mathbb{R}^N)$. Moreover, assume  that there exists a parabolic cylinder $Q=\Qla \Subset\Omega_T$ and a constant $\gamma\ge 1$ such that
 \begin{equation}\label{grad_stop_arg}
    \gamma^{-1}\lambda \le \biggl( \fiint_Q (|\na u|+|F|)^q\, dz \biggr)^\frac{1}{q} \le \gamma \lambda.
\end{equation}
Then there exists a constant $c=c(n,p,q,L,\gamma)$ such that
	\begin{equation}\label{p_poincare_ie2}
    \fiint_{Q}\frac{|u-u_{Q }|^{q}}{\rho^{q}}\,dz \le c\fiint_{Q}  ( |\na u| + |F| )^{q} \,dz.
    \end{equation}
\end{lemma}

\begin{proof}
    By substituting \eqref{grad_stop_arg} into \eqref{p_poincare_ie}, we obtain
    \begin{align*}
            &\fiint_{Q}\frac{|u-u_{Q }|^{q}}{\rho^{q}}\,dz 
             \le c\fiint_{Q}|\na u|^{q}\,dz \\
            &\qquad + c\biggl(  \fiint_{Q} ( |\na u|+ |F| )^{q}\,dz \biggr)^{2-p}   \biggl( \fiint_{Q} ( |\na u|+ |F| )^{p-1}\,dz \biggr)^q,
    \end{align*}
    where $c=c(n,p,q,L,\gamma)$.
    Therefore, by H\"older's inequality, we have
    \[
    \fiint_{Q}\frac{|u-u_{Q }|^{q}}{\rho^{q}}\,dz 
            \le c\fiint_{Q} (   |\na u| +  |F| )^{q} \,dz.
    \]
\end{proof}

The assumption \eqref{grad_stop_arg} naturally appears in the application of a stopping time argument in studying the gradient regularity of systems of parabolic PDEs, see \cite{MR2286632,MR1749438,KL_veryweak,MR2968162}.

\section{Self-improving properties of systems with $p$-growth}
The main challenge in self-improving properties of the parabolic Poincar\'e inequality is the appearance of the scaling factor in \eqref{p_poincare_ie}. We overcome this issue by applying the following scaling invariance property of \eqref{p_eq}, see \cite{MR2286632}. 

\begin{lemma}\label{p_scal_inv}
     Let $1< p<\infty$, $\max\{1,p-1\}\le q<\infty$ and $\la>0$.
    Assume that  $|F|\in L^{q}_{\loc}(\Omega_T)$ and that $u\in  L^{q}_{\loc}(0,T;W^{1,q}_{\loc} (\Om;\mathbb{R}^N))$ satisfies \eqref{p_weak_sol}
 for every $\varphi\in C_0^\infty(\Om_T;\mathbb{R}^N)$.
 For $\Qla(z_0)=B_\rho(x_0)\times I^\la_\rho(t_0)\Subset\Om_T$ with $\rho>0$ and $\la>0$, let $(y,s)\in Q_1=B_1(0)\times I_1(0)$ and
 \begin{equation}\label{p_scale}
 \begin{split}
     v(y,s)
        &=  (\la\rho)^{-1} u (x_0+ \rho y, t_0+ \la^{2-p} \rho^2 s ),\\
         \mathcal{B}( y, s,\xi)&= \la^{1-p} \mathcal{A}(x_0+ \rho y, t_0 + \la^{2-p}\rho^2s , \xi),\\
        G(y,s)&= \la^{-1} F (x_0+ \rho y, t_0+\la^{2-p}\rho^2s ).
 \end{split}
 \end{equation}
 Then $v\in  L^{q}_{\loc}(I_1(0) ;W^{1,q}_{\loc} (B_1(0);\mathbb{R}^N))$ satisfies 
 \[
 \iint_{Q_1} (-v\cdot \varphi_t+\mathcal{B}(y,s, \la \na v) \cdot \na \varphi) \,dy\,ds
=\iint_{Q_1} |G|^{p-2}G\cdot \na \varphi\,dy\,ds
 \]
 for every $\varphi\in C_0^\infty(Q_1 ;\mathbb{R}^N)$.
\end{lemma}

\begin{proof}
    We may assume that $x_0=0$ and $t_0=0$.
    Let $\varphi\in C_0^\infty(Q_1;\mathbb{R}^N)$ and let $\psi\in C_0^\infty(\Qla;\mathbb{R}^N)$ be such that, for $(y,s)\in Q_1$ and $(x,t)\in Q_\rho^\la$, we have
    \[
        \varphi(y,s)=\psi(\rho y, \la^{2-p}\rho^2 s )
        \quad\text{and}\quad 
        \psi(x,t)= \varphi(\rho^{-1}x, \la^{p-2}\rho^{-2}t).
    \]
    Note that 
  \[
    \frac{d \varphi}{ds} (y,s) = \la^{2-p}\rho^2 \frac{d\psi}{dt}(\rho y,  \la^{2-p}\rho^2 s )
 \]
and
\[
 \na_y \varphi(y,s) = \rho \na_x \psi ( \rho y,  \la^{2-p}\rho^2 s ).
\]
    It follows by scaling and \eqref{p_weak_sol} that
    \begin{align*}
        \begin{split}
            &\fiint_{Q_1} -v(y,s)\cdot \frac{d\varphi}{ds}(y,s) \,dy\,ds \\
            &\qquad= \la^{1-p}\rho \fiint_{Q_1} - u(\rho y, \la^{2-p}\rho^2 s)\cdot \frac{d\psi}{dt}( \rho y, \la^{2-p}\rho^2 s) \,dy\,ds\\
            &\qquad= \la^{1-p}\rho \fiint_{\Qla} -u\cdot \frac{d\psi}{dt} \,dx\,dt \\
            &\qquad= \la^{1-p}\rho \fiint_{\Qla} ( -\mathcal{A}(z,\na_x u)\cdot \na_x \psi+ |F|^{p-2}F\cdot \na_x \psi )\,dx\,dt.
        \end{split}
    \end{align*}
    We observe that
    \begin{align*}
            &\mathcal{A}(  \rho y, \la^{2-p}\rho^2 s , \na_x u ( \rho y, \la^{2-p}\rho^2 s) )\cdot \na_x \psi ( \rho y, \la^{2-p}\rho^2 s) \\
            &\qquad=\la^{p-1}\rho^{-1}\mathcal{B}(y,s, \la \na_y v (y,s)  ) \cdot \na_y \varphi(y,s).
    \end{align*}
    Similarly we have
    \begin{align*}
    &|F|^{p-2}F( \rho y, \la^{2-p}\rho^2 s) \cdot \na_x \psi ( \rho y, \la^{2-p}\rho^2 s)\\
    &\qquad= \la^{p-1}\rho^{-1} |G|^{p-2}G(y,s) \cdot \na_y\varphi (y,s).
    \end{align*}
    Therefore, again by scaling, we arrive at
    \begin{align*}
            &\fiint_{Q_1} -v\cdot \frac{d\varphi}{ds} \,dy\,ds\\
            &\qquad=\la^{1-p}\rho \fiint_{\Qla} ( -\mathcal{A}(z,\na_x u)\cdot \na_x \psi+ |F|^{p-2}F\cdot \na_x \psi )\,dx\,dt\\
            &\qquad=\fiint_{Q_1} ( -\mathcal{B}(y,s, \la \na_y v )\cdot \na_y \varphi+ |G|^{p-2}G\cdot \na_y \varphi )\,dy\,ds.
    \end{align*}
    This completes the proof.
\end{proof}

\begin{theorem}\label{p_main_left}
Let $1< p<\infty$, $\max\{1,p-1\}\le q<\infty$ and $\la\ge 1$. 
Assume that  $|F|\in L^{q}_{\loc}(\Omega_T)$ and that $u\in  L^q_{\loc}(0,T;W^{1,q}_{\loc} (\Om;\mathbb{R}^N))$ satisfies \eqref{p_weak_sol} for every $\varphi\in C_0^\infty(\Om_T;\mathbb{R}^N)$ and that there exist a parabolic cylinder $Q=\Qla \Subset\Omega_T$ and a constant $\gamma>1$ such that
\begin{equation}\label{p_main_left_stopping}
   \gamma^{-1}\lambda \le \biggl(\fiint_Q (|\na u|+|F|)^q\, dz\biggr)^\frac{1}{q} \le \lambda.
\end{equation}
\begin{itemize}
    \item[(i)] If $2< p<\infty$ and $p-1 \le q <n+p$, then for $ q\le q^*<\tfrac{q(n+2)}{n+p-q}$, there exists a constant $c=c(n,p,q,q^*,L,\gamma)$ such that 
    \begin{equation}\label{p_poincare_ieq}
    \left( \fiint_{Q } \frac{|u-u_{Q}|^{q^*}}{\rho^{q^*}}\,dz\right)^\frac{1}{q^*}
    \le c\biggl( \fiint_{Q } ( |\na u| + |F| ) ^q\,dz\biggr)^\frac{1}{q}.
    \end{equation}
    If $1<p<2$ and $\max\bigl\{ 1,\tfrac{(2-p)n}{2}\bigr\} < q <n+p$, then the estimate above holds for $q\le q^*<\tfrac{q(n+2)}{n+p-q}$.
    \item[(ii)] If $2< p<\infty$ and $q=n+p$, then there exists a constant $c=c(n,p,L,\gamma)$ such that
    \begin{equation}\label{p_poincare_ieq_2}
        \fiint_Q\exp\left(\left(\frac{|u-u_{Q}|}{\displaystyle c\rho\left(\fiint_{Q}(|\na u|+|F|)^{n+p} \,dz\right)^{\frac{1}{n+p}}}\right)^{\frac{n+2}{n+p}} \right)\,dz\le c.
    \end{equation}
    If $1<p<2$ and $q=n+p$, then there exists a constant $c=c(n,p,L,\gamma)$ such that
    \begin{equation}\label{p_poincare_ieq_22}
        \fiint_Q\exp\left(
   \frac{|u-u_{Q}|}{\displaystyle c\rho\left(\fiint_{Q}(|\na u|+|F|)^{n+p} \,dz\right)^{\frac{1}{n+p}}} \right)\,dz\le c.
    \end{equation}
    \item[(iii)] If $q>n+p$, then there exists a constant $c=c(n,p,q,L,\gamma)$ such that
    \begin{equation}\label{p_poincare_ieq_3}
        \esssup_Q\frac{|u-u_{Q}|}{\rho}
    \le c\biggl( \fiint_{Q } ( |\na u| + |F| ) ^q\,dz\biggr)^\frac{1}{q}.
    \end{equation}
\end{itemize}
\end{theorem}
Note that if $\tfrac{2n}{n+2}<p<2$, then we may choose $q=p$ in \eqref{p_poincare_ieq}. On the other hand, estimates in (ii) and (iii) recover corresponding estimates in Theorem~\ref{main_left}. However, these estimates hold only for cylinders that satisfy \eqref{p_main_left_stopping} and they do not imply BMO or H\"older continuity properties as in Remark~\ref{main_left_rmk}.

\begin{proof}[Proof of \eqref{p_poincare_ieq}]
We consider several cases and follow the approach in \cite{MR3895752}.

\textit{Case~1}: $2<p<\infty$.
    We apply Lemma~\ref{p_scal_inv} with the notation \eqref{p_scale} and have
    \[
     \fiint_{Q_1}  (|\na v|+|G|)^q\,dz = \frac{1}{\la^q}\fiint_{Q} (|\na u|+|F|)^q\,dz\le 1,
    \]
    where  $Q_1$ is as in Lemma~\ref{p_scal_inv}.
    For $\La>1$, let $k\ge0$ be an integer such that $2^k<\La\le 2^{k+1}$.
     We may assume that
\begin{equation}\label{p_poincare_ieq_lv}
    \cs=
    \Biggl( \frac{2^{n+3p+2}c}{1-\left( \frac{1}{2} \right)^\frac{p-2}{q} } \Biggr)^\frac{q}{q-p+2} \le 2^k,
\end{equation}
where $c=c(n,p,L)>0$ is the constant in Lemma~\ref{p_poincare_lem}.
For every $w\in \{ z\in Q_1: |v(z)-v_{Q_1}|>\La  \}$ we consider a sequence consisting of two types of intrinsic cylinders. Let 
\[
\bigl\{ Q_i=Q_{1}^{2^i}(z_i) \bigr\}_{i=2}^{ k } 
\quad\text{and}\quad 
\bigl\{Q_i=Q_{\rho_i}^{\frac{\La}{\rho_i}}(z_i) \bigr\}_{i=k+1}^\infty
\]
be such that
\[
 \rho_i=2^{-i+k}, \quad Q_{i+1}\subset Q_i 
 \quad\text{and}\quad \lim_{i\to\infty} v_{Q_i}=v(w).
\]
 By Lemma~\ref{p_poincare_lem} we get
    \begin{align*}
            \La
            &< |v(w)-v_{Q_1}|
            \le \sum_{i=1}^\infty |v_{Q_{i+1}}-v_{Q_i}|
            \le \sum_{i=1}^\infty\frac{|Q_i|}{|Q_{i+1}|}\fiint_{Q_i}|v-v_{Q_i}|\,dz \\
            &\le  2^{n+3p} c \sum_{i=1}^k \biggl( \biggl(\fiint_{Q_i} g^q\,dz \biggr)^\frac{1}{q}+ 2^{i(2-p)}\fiint_{Q_i} g^{p-1}  \,dz  \biggr) \\
            &\qquad + 2^{n+p}c \sum_{i=k+1}^\infty \rho_i \biggl( \biggl(\fiint_{Q_i} g^q\,dz \biggr)^\frac{1}{q}+ \La^{2-p}\rho_i^{p-2} \fiint_{Q_i} g^{p-1}  \,dz  \biggr),
    \end{align*}
where $g=|\na v|+|G|$. 
Since $ (g^q)_{Q_1}\le 1$ and \eqref{p_poincare_ieq_lv} holds, we have
\begin{align*}
    &2^{n+3p} c \sum_{i=1}^k \biggl( \biggl(\fiint_{Q_i} g^q\,dz \biggr)^\frac{1}{q}+ 2^{i(2-p)}\fiint_{Q_i} g^{p-1}  \,dz  \biggr)\\
    &\qquad\le 2^{n+3p} c \sum_{i=1}^k \biggl( 2^\frac{(p-2)i}{q}\biggl(\fiint_{Q_1} g^q\,dz \biggr)^\frac{1}{q}+ \fiint_{Q_1} g^{p-1}  \,dz  \biggr)\\
    &\qquad\le 2^{n+3p+1}c \sum_{i=1}^k 2^{\frac{(p-2)i}{q} } 
    \le \Biggl( \frac{2^{n+3p+1}c}{1-\left( \frac{1}{2} \right)^\frac{p-2}{q} } \Biggr) 2^{\frac{(p-2)k}{q}} \\
    &\qquad\le \Biggl( \frac{2^{n+3p+1}c}{1-\left( \frac{1}{2} \right)^\frac{p-2}{q} } \Biggr) 2^{\frac{(p-2-q)k}{q}} \La \le \frac{1}{2}\La.
\end{align*}
Therefore we obtain
\[
 \La \le c \sum_{i=k+1}^\infty \rho_i \biggl( \biggl(\fiint_{Q_i} g^q\,dz \biggr)^\frac{1}{q}+ \La^{2-p}\rho_i^{p-2} \fiint_{Q_i} g^{p-1}  \,dz  \biggr).
\]
Let $0<\ep<1$ be determined later and note that
\begin{align*}
    &  \min\{ (1-2^{-\ep}), (1-2^{-(p-1)\ep}) \} \sum_{i=k+1}^\infty( 2^{-\ep(i-k)} + 2^{-(p-1)\ep(i-k)}   )\La \\
    &\qquad\le 2\La 
    \le c \sum_{i=k+1}^\infty \rho_i \biggl( \biggl(\fiint_{Q_i} g^q\,dz \biggr)^\frac{1}{q}+ \La^{2-p}\rho_i^{p-2} \fiint_{Q_i} g^{p-1}  \,dz  \biggr).
\end{align*}
Comparing each term in the summations, there exists $i_w\ge k+1$ such that
\begin{align*}
    &  ( 2^{-\ep(i_w-k)} + 2^{-(p-1)\ep(i_w-k)}   )\La \\
    &  \qquad\le c_\ep \biggl( \rho_{i_w} \biggl(\fiint_{Q_{i_w}} g^q\,dz \biggr)^\frac{1}{q}+ \La^{2-p}\rho_{i_w}^{p-1} \left(\fiint_{Q_{i_w}} g^q\,dz \right)^\frac{p-1}{q}  \biggr),
\end{align*}
where $c_\ep=c_\ep(n,p,L,\ep)$. If the first term on the right-hand side is bigger than the second term, we have
\begin{equation}\label{p_poincare_ieq_cal_1}
    2^{-\ep(i_w-k)}\La \le c_\ep \rho_{i_w} \biggl(\fiint_{Q_{i_w}} g^q\,dz \biggr)^\frac{1}{q}.
\end{equation}
On the other hand, if the second term is bigger than the first term, then we have
\[
2^{-(p-1)\ep(i_w-k)}\La \le c_\ep\La^{2-p}\rho_{i_w}^{p-1} \biggl(\fiint_{Q_{i_w}} g^q\,dz \biggr)^\frac{p-1}{q}.
\]
Thus in both cases \eqref{p_poincare_ieq_cal_1} holds. Since $\rho_{i_w}=2^{-i_w+k}$, we have
\[
\La^{q + \frac{q(\ep-1)(p-2)}{n+p} } |Q_{i_w}|^{1+\frac{(\ep -1 )q}{n+p}} = \La^q \rho_{i_w}^{q(\ep-1)} |Q_{i_w}| \le c_\ep \iint_{Q_{i_w}} g^q\,dz.
\]
Denoting
\[
q^*_\ep =q\left( 1+ \frac{(\ep-1)(p-2)}{n+p} \right) \left( 1+\frac{(\ep-1)q}{n+p} \right)^{-1},
\]
the above inequality becomes
\begin{equation}\label{p_poincare_ieq_cal2}
    \La^{q^*_\ep} |Q_{i_w}|\le c_\ep \biggl( \iint_{Q_{i_w}} g^q\,dz\biggr)^{ \left( 1+\frac{(\ep -1 )q}{n+p} \right)^{-1} }.
\end{equation}

Then we apply a standard covering lemma. Since \eqref{p_poincare_ieq_cal2} holds for every point in $\{ z\in Q_1: |v(z)-v_{Q_1}|>\La  \}$, there exists a countable disjoint subfamily $\{\Q_j\}_{j=1}^\infty$  of $\{Q_{i_w}:w\in \{ z\in Q_1: |v(z)-v_{Q_1}|>\La  \} \}$ such that the union of $\{5\Q_j\}_{j=1}^\infty$ covers $\{ z\in Q_1: |v(z)-v_{Q_1}|>\La  \}$. Therefore, for $\La> 2^k\ge \cs$ with non-negative integer $k$ and a constant $\cs$ defined in \eqref{p_poincare_ieq_lv} we get
    \[
    \La^{ q^*_\ep  } | \{ z\in Q_1: |v(z)-v_{Q}|>\La  \} |
     \le c_\ep \left( \iint_{Q_1} g^q\,dz \right)^{\left( 1+\frac{(\ep-1)q}{n+p} \right)^{-1}}\le c_\ep.
    \]
Here we again used $(g^q)_{Q_1}\le 1$. For the case $0<\La \le 2\cs$ we trivially have
\[
 \La^{ q^*_\ep  } | \{ z\in Q_1: |v(z)-v_{Q}|>\La  \} |
     \le c_\ep
\]
and thus we obtain
\[
 \sup_{\La>0} \La^{ q^*_\ep  } | \{ z\in Q_1: |v(z)-v_{Q}|>\La  \} |
     \le c_\ep.
\]

For given $q^*$ we take $0<\ep<1$ small enough so that $q^*< q^*_\ep$. Then we have
    \[
    \biggl( \fiint_{Q_1 } |v-v_{Q_1}|^{q^*}\,dz  \biggr)^\frac{1}{q^*}
            \le c \sup_{\La>0} \La| \{ z\in Q_1: |v(z)-v_{Q_1}|>\La  \} |^\frac{1}{q^*_\ep}\le c,
    \]
    where $c=c(n,p,q,q^*,L)$. 
    By scaling back the above inequality and using \eqref{p_main_left_stopping}, we obtain
   \[
    \biggl( \fiint_{Q } \frac{|u-u_{Q}|^{q^*}}{\rho^{q^*}} \,dz  \biggr)^\frac{1}{q^*}\le c\la \le c\biggl(  \fiint_{Q} (|\na u|+|F|)^q\,dz \biggr)^\frac{1}{q},
    \]
    where $c=c(n,p,q,q^*,L,\gamma)$.
    This completes the proof.

\textit{Case~2}: $1<p<2$.
In this case we modify the proof of the previous case.
For $\La>0$, let $k\ge0$ be an integer such that $2^k<\La\le 2^{k+1}$. 
We may assume that
\begin{equation}\label{p_poincare_ieq_lv2}
    \cs=
    \Biggl( \frac{ 2^{n+p+2} }{ 1-\left( \frac{1}{2} \right)^ \frac{(2-p)n}{ 2 q} } \Biggr)^{ \left(  \frac{2-p}{ p } \left(  -\frac{n}{q}+1 +\frac{p}{2-p} \right)  \right)^{-1} } \le 2^k,
\end{equation}
where $c=c(n,p,L)>0$ is the constant in Lemma~\ref{p_poincare_lem}.
Note that the condition on $q$ guaranties that $-\tfrac{n}{q}+1 +\tfrac{p}{2-p}>0$.

We slightly abuse the notation and denote by $\tfrac{2k}{p}$ the largest integer smaller than $\tfrac{2k}{p}$.
For every $w\in \{ z\in Q_1: |v(z)-v_{Q_1}|>\La  \}$ let 
\[
\Bigl\{Q_i= Q_{2^\frac{(p-2)i}{2}} ^{2^i} (z_i) \Bigr\}_{i=2}^{\frac{2k}{p}}
\quad\text{and}\quad 
\bigl\{Q_i=Q_{\rho_i}^{\frac{\La}{2\rho_i} }(z_i)\bigr\}_{i=\frac{2k}{p}+1}^\infty
\]
such that
\[
 \rho_i=2^{-i+k}, \quad Q_{i+1}\subset Q_i 
 \quad\text{and}\quad \lim_{i\to\infty} v_{Q_i}=v(w).
\]
 Again we apply Lemma~\ref{p_poincare_lem} to have
    \begin{align*}
            \La
            &\le   2^{n+p}c \sum_{i=1}^{\frac{2k}{p}} 2^{\frac{(p-2)i}{2}} \biggl( \biggl(\fiint_{Q_i} g^q\,dz\biggr)^\frac{1}{q}+  \fiint_{Q_i} g^{p-1}  \,dz  \biggr)\\
            &\qquad+   2^{n+p}c \sum_{i=\frac{2k}{p}+1}^\infty \rho_i \biggl( \biggl(\fiint_{Q_i} g^q\,dz\biggr)^\frac{1}{q}+ \La^{2-p} \rho_i^{p-2} \fiint_{Q_i} g^{p-1}  \,dz  \biggr).
    \end{align*}
Note that
\begin{align*}
    &2^{n+p}c \sum_{i=1}^{\frac{2k}{p}} 2^{\frac{(p-2)i}{2}} \biggl( \biggl(\fiint_{Q_i} g^q\,dz\biggr)^\frac{1}{q}+  \left( \fiint_{Q_i} g^q  \,dz \right)^\frac{p-1}{q}  \biggr)\\
    &\qquad\le 2^{n+p+1}c \sum_{i=1}^{\frac{2k}{p}} 2^{\frac{(p-2)i}{2} +\frac{(2-p)n i}{2q} }
    \le \Biggl(  \frac{2^{n+p+1}c}{ 1-\left(  \frac{1}{2} \right)^{  \frac{(2-p)n}{ 2 q}  } }  \Biggr) 2^{   \frac{2-p}{ 2 } \left(  \frac{n}{q}-1 \right)  \frac{2k}{p}  }\\
    &\qquad\le\Biggl(  \frac{2^{n+p+2}c}{ 1-\left(  \frac{1}{2} \right)^{  \frac{(2-p)n}{ 2 q} } }  \Biggr) 2^{   \frac{2-p}{ p } \left(  \frac{n}{q}-1 -\frac{p}{2-p} \right)  k  } \frac{1}{2}\La
    \le\frac{1}{2}\La,
\end{align*}
where we used the fact that $(g^q)_{Q_1}\le 1$ and \eqref{p_poincare_ieq_lv2}. It follows that
\[
 2\La \le c \sum_{i=\frac{2k}{p}+1}^\infty \rho_i \biggl( \biggl(\fiint_{Q_i} g^q\,dz\biggr)^\frac{1}{q}+ \La^{2-p} \rho_i^{p-2} \fiint_{Q_i} g^{p-1}  \,dz  \biggr)
\]
We again break the left-hand side into a summation
\[
 \min\{ (1-2^{-\ep}), (1-2^{-(p-1)\ep}) \} \sum_{i=\frac{2k}{p}+1}^\infty( 2^{-\ep(i-k)} + 2^{-(p-1)\ep(i-k)}   )\La \le 2\La
\]
and conclude that there exists $i_w>\tfrac{2k}{p}$ such that
\[
  2^{-\ep(i_w-k)}\La \le c_\ep \rho_{i_w}\biggl(\fiint_{Q_{i_w}} g^q\,dz\biggr)^\frac{1}{q}.
\]
Recall that $\rho_{i_w}=2^{-i_w+k}$ and $|Q_{i_w}|= c2^{(-i_w+k)(n+p)}\La^{2-p}$, where $c=c(n)$. We obtain
\[
 \La^{ q + \frac{q (\ep-1)(p-2) }{n+p} } |Q_{i_w}|^{1+\frac{(\ep-1)q}{n+p}} 
 \le c_\ep \iint_{Q_{i_w}} g^q\,dz.
\]
Moreover, as in the previous case, we get by the Vitali covering lemma that
\[
\La^{q^*_\ep} | \{ z\in Q_1: |v(z)-v_{Q}|>\La  \} |
     \le c_\ep ,
\]
where
\[
 q^*_\ep= q\left( 1 + \frac{(\ep-1)(p-2) }{n+p} \right) \left( 1+\frac{(\ep-1)q}{n+p} \right)^{-1}.
\]
Following the same argument as in the previous case, the conclusion follows. We omit the details.
\end{proof}

\begin{proof}[Proof of \eqref{p_poincare_ieq_2} and \eqref{p_poincare_ieq_22}]
    We again apply Lemma~\ref{p_scal_inv} as in the proof of \eqref{p_poincare_ieq} by setting $\La=|v(z)-v_{Q_1}|$ with a condition that if $2<p<\infty$ and $\cs<2^k< |v(z)-v_{Q_1}|\le 2^{k+1}$, where $k$ is an non-negative integer and $\cs$ is defined in \eqref{p_poincare_ieq_lv}  and if $1<p<2$ and $\cs<2^k<|v(z)-v_{Q_1}|\le 2^{k+1}$, where $k$ is an non-negative integer and $\cs$ is defined in \eqref{p_poincare_ieq_lv2}.
    
    First we consider $2<p<\infty$.
    Employing Lemma~\ref{p_poincare_lem} we get
  \begin{align*}
  |v(z)-v_{Q_1}| 
          &\le c \sum_{i=k+1}^\infty \rho_i\fiint_{Q_i} g\,dz   \\
          &\qquad+ c|v(z)-v_{Q_1}|^{2-p}\sum_{i=k+1}^\infty \rho_i^{p-1} \fiint_{Q_i} g^{n+p}\,dz,
  \end{align*}
  where $c=c(n,p,L)$. 
  By comparing the first summation and the second summation on the right-hand side, either of the following holds
  \begin{equation}\label{p_poincare_cal3}
      |v(z)-v_{Q_1}| 
          \le c \sum_{i=k+1}^\infty \rho_i \fiint_{Q_i} g \,dz
  \end{equation}
  or
  \begin{equation}\label{p_poincare_cal4}
      |v(z)-v_{Q_1}|^{p-1} 
  \le c\sum_{i=k+1}^\infty \rho_i^{p-1} \fiint_{Q_i} g^{p-1} \,dz .
  \end{equation}

  We fix a positive real number $m>n+p$ and first consider \eqref{p_poincare_cal3}. Recall that $|Q_i|=c\La^{2-p}\rho_i^{n+p}$ and $\rho_i=2^{-i+k}$ for $i\ge k+1$. We estimate the right-hand side of \eqref{p_poincare_cal3} as
  \begin{align*}
      &\rho_i \fiint_{Q_i} g \,dz
      = \rho_i \biggl( \fiint_{Q_i} g \,dz\biggr)^{1-\frac{n+p}{m}} \left( \fiint_{Q_i} g \,dz\right)^{\frac{n+p}{m}}\\
      &\qquad\le \rho_i \biggl( \fiint_{Q_i} g^{n+p} \,dz\biggr)^{\frac{1}{n+p}-\frac{1}{m}} (M^*(g^{p-1} 1_{Q_1}))^{\frac{n+p}{m(p-1)}}\\
      &\qquad\le c 2^{\frac{(-i+k)(n+p)}{m} } \La^{ \frac{p-2}{n+p} - \frac{p-2}{m} }
       \biggl( \iint_{Q_i} g^{n+p} \,dz\biggr)^{\frac{1}{n+p}-\frac{1}{m}}  (M^*(g^{p-1} 1_{Q_1}))^{\frac{n+p}{m(p-1)}}.
  \end{align*}
  Since $(g^{n+p})_{Q_1} \le 1$ and $|Q_1|=c$, where $c=c(n)$, we deduce that
  \[
  |v(z)-v_{Q_1}| \le c 2^{-\frac{n+p}{m}} \La^{ \frac{p-2}{n+p} - \frac{p-2}{m} } (M^*(g^{p-1} 1_{Q_1}))^{\frac{n+p}{m(p-1)}}.
  \]
  As we set $\La=|v(z)-v_{Q_1}|>1$ and $p>2$, the above inequality becomes
  \begin{equation}\label{p_poincare_cal5}
  \begin{split}
       |v(z)-v_{Q_1}|^\frac{n+2}{n+p} 
       &<|v(z)-v_{Q_1}|^{1-\frac{p-2}{n+p} + \frac{p-2}{m} } \\
       &\le c 2^{-\frac{n+p}{m}}  (M^*(g^{p-1} 1_{Q_1}))^{\frac{n+p}{m(p-1)}}.
    \end{split}
  \end{equation}
  On the other hand if \eqref{p_poincare_cal4} holds, then we have
  \begin{align*}
      \rho_i^{p-1} \fiint_{Q_i} g^{p-1} \,dz
      &\le \rho_i^{p-1} \biggl( \fiint_{Q_i} g^{p-1} \,dz\biggr)^{1-\frac{n+p}{m}} \biggl( \fiint_{Q_i} g^{p-1} \,dz\biggr)^{\frac{n+p}{m}}\\
      &\le c 2^{ {\frac{-i+k}{m} } (n+p)(p-1)  }\La^{ \left( \frac{p-2}{n+p}-\frac{p-2}{m}\right) (p-1) }\\
      &\qquad\times 
       \biggl( \iint_{Q_i} g^{n+p} \,dz\biggr)^{\frac{p-1}{n+p}-\frac{p-1}{m}}  (M^*(g^{p-1} 1_{Q_1}))^{\frac{n+p}{m}}.
  \end{align*}
  Therefore \eqref{p_poincare_cal5} again holds. 
  Then by \eqref{p_poincare_cal5} we get
  \begin{align*}
      \iint_{Q_1} |v-v_{Q_1}|^{\frac{(n+2)m}{n+p}} \,dz 
      &\le \iint_{ \{z\in Q_1: |v(z)-v_{Q_1}|\le 2\cs \} } |v-v_{Q_1}|^{\frac{(n+2)m}{n+p}} \,dz  \\
      &\qquad + \iint_{ \{z\in Q_1: |v(z)
      -v_{Q_1}|> \cs \} } |v-v_{Q_1}|^{\frac{(n+2)m}{n+p}} \,dz \\
      &\le c^m +c^m \iint_{Q_1} (M^*(g^{p-1} 1_{Q_1}))^\frac{n+p}{p-1} \,dz.
  \end{align*}
  Furthermore, by the maximal function theorem for the strong maximal function and the fact $(g^{n+p})_{Q_1}\le 1$, we obtain
  \[
      \fiint_{Q_1} |v-v_{Q_1}|^{\frac{(n+2)m}{n+p}} \,dz \le c^m.
  \]
  For the case $0<m\le n+p$ we may apply \eqref{p_poincare_ie2} to have the above inequality. Thus the above inequality holds for any $m>0$ with $c=c(n,p,L)$.

  By the power series representation of the exponential function and by applying the inequality above, we obtain
  \begin{align*}
  \fiint_{Q_1} \exp \big(  |v-v_{Q_1}|^{\frac{n+2}{n+p}}  \big) \,dz 
      &\le \sum_{m=0}^\infty \frac{1}{m!}\fiint_{Q_1} |v-v_{Q_1}|^{\frac{(n+2)m}{n+p}} \,dz\\
      &\le \sum_{m=0}^\infty \frac{c^m}{m!} \le c.
  \end{align*}
  Hence the proof of \eqref{p_poincare_ieq_2} is completed by scaling back the left-hand side and using \eqref{p_main_left_stopping}.

  The case $1<p<2$ is analogous. The difference arises in the exponent of the second term of \eqref{p_poincare_cal5}. Since $1-\tfrac{p-2}{n+p}+\tfrac{p-2}{m}>1$ holds for $m>n+p$, we have
  \[
  |v(z)-v_{Q_1}|\le c 2^{-\frac{n+p}{m}}  (M^*(g 1_{Q_1}))^{\frac{n+p}{m}}.
  \]
  Therefore we get
  \[
  \fiint_{Q_1} |v-v_{Q_1}|^m \,dz \le c^m
  \]
  and thus 
  \[
  \fiint_{Q_1} \exp (  |v-v_{Q_1}|  ) \,dz \le c.
  \]
  The desired estimate follows by scaling back the left-hand side and using \eqref{p_main_left_stopping}.
\end{proof}

\begin{proof}[Proof of \eqref{p_poincare_ieq_3}]
We only consider $2< p<\infty$, since the case $1<p<2$ is analogous.
    We follow the argument in the proof of \eqref{p_poincare_ieq_2} and \eqref{p_poincare_ieq_22} to arrive at the point where either \eqref{p_poincare_cal3} or \eqref{p_poincare_cal4} holds.
    Suppose \eqref{p_poincare_cal3} holds first. Then if $|v(z)-v_{Q_1}|>\cs$, then we apply H\"older's inequality to have
    \begin{align*}
        |v(z)-v_{Q_1}|
        &\le c \sum_{i=k+1}^\infty \rho_i \fiint_{Q_i} g \,dz \\
        &\le c \sum_{i=k+1}^\infty \rho_i^{1-\frac{n+p}{q}} |v(z)-v_{Q_1}|^\frac{p-2}{q} \biggl(\iint_{Q_i} g^q \,dz\biggr)^\frac{1}{q}\\
        &\le c \sum_{i=k+1}^\infty 2^{ (-i+k) \left(  1-\frac{n+p}{q}  \right) } |v(z)-v_{Q_1}|^\frac{p-2}{q} \biggl(\iint_{Q_1} g^q \,dz\biggr)^\frac{1}{q}\\
        &\le c|v(z)-v_{Q_1}|^\frac{p-2}{q}.
    \end{align*}
    Therefore we have $|v(z)-v_{Q_1}|\le c$, where $c=c(n,p,q,L)$. A similar calculation leads to the same conclusion for the case \eqref{p_poincare_cal4}. 
    The remaining case can be covered by replacing the constant by $2\cs$. By scaling back, the conclusion follows. This completes the proof.
\end{proof}

\begin{remark}
    In the proof above, the first inequality in \eqref{p_main_left_stopping} was used at the last stage when we apply scaling back to replace $\lambda$ by
    \[
    \biggl(\fiint_Q (|\na u|+|F|)^q\, dz\biggr)^\frac{1}{q}.
    \]
    Leaving $\lambda$ in the scaling back process, the estimate still holds with only the second inequality in \eqref{p_main_left_stopping}.
\end{remark}

Next we discuss the consequence of Theorem~\ref{p_main_left}. Note that if $\tfrac{2n}{n+2}< p<\infty$ and \eqref{p_main_left_stopping} holds for $q<p$ sufficiently close to $p$, then weak solutions belongs to $L^2$.
To begin with, we present a Caccioppoli inequality in the intrinsic geometry.

\begin{lemma}
	Let $\tfrac{2n}{n+2}< p<\infty$.
    Assume that $|F|\in L^p_{\loc}(\Omega_T)$ and that $u\in  L^p_{\loc}(0,T;W^{1,p}_{\loc} (\Om;\mathbb{R}^N))$ is a weak solution to \eqref{p_eq}. 
    Moreover, assume that there exist a parabolic cylinder $Q_{2\rho}^\la=Q_{2\rho}^\la(z_0)=B_{2\rho}(x_0)\times I^\la_{2\rho}(t_0)\Subset\Om_T$ and an exponent $1<q\le p$ satisfying $2<\tfrac{q(n+2)}{n+p-q}$ such that
    \begin{equation}\label{cacc_stopping}
    \biggl(\fiint_{Q_{2\rho}^\la} (|\na u|+|F|)^q \, dz\biggr)^\frac{1}{q} \le \lambda.
    \end{equation}
    Then $u\in L^\infty (I_{2\rho}(t_0);L^2(B_{2\rho}(x_0);\mathbb{R}^N))$.
    Moreover, there exists a constant $c=c(n,N,p,\nu,L)$ such that, for any $0<\rho<\rho_1<\rho_2< 2\rho$, we have
    \begin{equation}\label{p_cacc_ie}
	\begin{split}
			&\esssup_{t\in I^\la_{\rho_1}(t_0)}\la^{p-2}\fint_{B_{\rho_1}(x_0)}\frac{|u-u_{Q^\la_{\rho_1}(z_0)}|^2}{\rho_1^2}\,dx+\fiint_{ Q^\la_{\rho_1}(z_0)}|\na u|^p\,dz\\
			&\qquad\le c\fiint_{Q^\la_{\rho_2}(z_0)} \frac{|u-u_{Q^\la_{\rho_2}(z_0)}|^p}{(\rho_2-\rho_1)^p}\,dz + c\la^{p-2}\fiint_{Q^\la_{\rho_2}(z_0)} \frac{|u-u_{Q^\la_{\rho_2}(z_0)}|^2}{(\rho_2-\rho_1)^2}\,dz  \\
            &\qquad\qquad+c\fiint_{Q^\la_{\rho_2}(z_0)}  |F|^p  \,dz.
	\end{split}
    \end{equation}
\end{lemma}

As in the linear growth case, we can combine \eqref{p_poincare_ieq} and \eqref{p_cacc_ie} to derive a reverse H\"older inequality for the mean oscillation.

\begin{theorem}\label{p_rev_u}
    Let $\tfrac{2n}{n+2}< p<\infty$.
    Assume that $|F|\in L^p_{\loc}(\Omega_T)$ and that $u\in  L^p_{\loc}(0,T;W^{1,p}_{\loc} (\Om;\mathbb{R}^N))$ is a weak solution to \eqref{p_eq}. 
    Moreover, assume that there exist a parabolic cylinder $2Q=Q_{2\rho}^\la\Subset\Om_T$, an exponent $1<q\le p$ satisfying $\max\{2,p\}<\tfrac{q(n+2)}{n+p-q}$ and a constant $\gamma>1$ such that
    \[
    \gamma^{-1}\lambda \le \biggl(\fiint_{ \alpha Q } (|\na u|+|F|)^q \, dz\biggr)^\frac{1}{q} \le \lambda
    \]
    for any $\alpha$ with $1\le \alpha \le 2$.
    Then for any $q^*$ with $\max\{2,p\}<q^*<\tfrac{q(n+2)}{n+p-q}$ there exists a constant $c=c(n,N,p,q,q^*,\nu,L,\gamma)$ such that
    \[
 \biggl( \fiint_{Q} \frac{|u-u_{Q}|^{q^*}}{\rho^{q^*}}\,dz\biggr)^\frac{1}{q^*}
 \le c \fiint_{2Q} \frac{|u-u_{2Q}|}{2\rho}\,dz
    +c\fiint_{2Q}  |F|^p  \,dz.
\]
\end{theorem}

\begin{proof}
For $0<\rho<\rho_1<\rho_2< 2\rho$, we denote $Q_1=\tfrac{\rho_1}{\rho} Q$ and $Q_2=\tfrac{\rho_2}{\rho}Q$. We apply \eqref{p_cacc_ie} to have
\begin{align*}
\fiint_{Q_1} |\na u|^p\,dz 
    &\le  c\fiint_{Q_2} \frac{|u-u_{Q_2}|^p}{(\rho_2-\rho_1)^p}\,dz 
    + c\la^{p-2}\fiint_{Q_2} \frac{|u-u_{Q_2}|^2}{(\rho_2-\rho_1)^2}\,dz \\
    &\qquad+c\fiint_{Q_2}  |F|^p  \,dz.
\end{align*}
We consider several cases.

\textit{Case~1}: $2< p<\infty$.
 We use Young's inequality for the second term on the right-hand side to have
 \begin{align*}
     \fiint_{Q_1} |\na u|^p\,dz 
    &\le \frac{1}{2\gamma^p}\la^p + c\fiint_{Q_2} \frac{|u-u_{Q_2}|^p}{(\rho_2-\rho_1)^p}\,dz
    +c\fiint_{Q_2}  |F|^p  \,dz\\
    &\le \frac{1}{2}\fiint_{Q_1} |\na u|^p\,dz + c\fiint_{Q_2} \frac{|u-u_{Q_2}|^p}{(\rho_2-\rho_1)^p}\,dz
    +c\fiint_{Q_2}  |F|^p  \,dz.
 \end{align*}
By \eqref{p_poincare_ieq} we get
\[
 \biggl( \fiint_{Q_1} \frac{|u-u_{Q_1}|^{q^*}}{\rho_1^{q^*}}\,dz\biggr)^\frac{p}{q^*}
 \le c\fiint_{Q_2} \frac{|u-u_{Q_2}|^p}{(\rho_2-\rho_1)^p}\,dz
    +c\fiint_{Q_2}  |F|^p  \,dz,
\]
where $c=c(n,N,p,q,q^*,\nu,L,\gamma)$. An interpolation inequality gives
\[
 \fiint_{Q_2} \frac{|u-u_{Q_2}|^p}{\rho_2^p}\,dz
 \le \biggl( \fiint_{Q_2} \frac{|u-u_{Q_2}|}{\rho_2}\,dz \biggr)^{p\theta}
 \biggl( \fiint_{Q_2} \frac{|u-u_{Q_2}|^{q^*}}{\rho_2^{q^*}}\,dz \biggr)^\frac{p-p\theta}{q^*},
\]
where $0<\theta<1$ satisfies $\tfrac{1}{p}=\theta+\tfrac{1-\theta}{q^*}$. 
Again, by Young's inequality, we obtain
\begin{align*}
    &\biggl( \fiint_{Q_1} \frac{|u-u_{Q_1}|^{q^*}}{\rho_1^{q^*}}\,dz\biggr)^\frac{p}{q^*}
    \le \frac{1}{2}\biggl( \fiint_{Q_2} \frac{|u-u_{Q_2}|^{q^*}}{\rho_2^{q^*}}\,dz\biggr)^\frac{p}{q^*} \\
    &\qquad + c \left( \frac{\rho_2}{\rho_2-\rho_1} \right)^\frac{p}{1-\theta} \biggl( \fiint_{Q_2} \frac{|u-u_{Q_2}|}{\rho_2}\,dz\biggr)^p
    +c\fiint_{Q_2}  |F|^p  \,dz.
\end{align*}
Hence the conclusion follows by absorbing terms. 

\textit{Case~2}: $\tfrac{2n}{n+2}<p<2$. 
In this case we observe that
\begin{align*}
    \left( \fiint_{Q_1} |\na u|^q\,dz \right)^\frac{2}{q}
     &\le \la^{2-p}\fiint_{Q_1} |\na u|^p\,dz \\
    &\le  c\la^{2-p}\fiint_{Q_2} \frac{|u-u_{Q_2}|^p}{(\rho_2-\rho_1)^p}\,dz 
    + c\fiint_{Q_2} \frac{|u-u_{Q_2}|^2}{(\rho_2-\rho_1)^2}\,dz \\
    &\qquad+c\la^{2-p}\fiint_{Q_2}  |F|^p  \,dz.
\end{align*}
Applying Young's inequality to the first and the third terms on the right-hand side, we have
\[
\biggl( \fiint_{Q_1} |\na u|^q\,dz \biggr)^\frac{2}{q}
    \le \frac{1}{2\gamma^2}\la^2 + c\fiint_{Q_2} \frac{|u-u_{Q_2}|^2}{(\rho_2-\rho_1)^2}\,dz 
    +c\biggl(\fiint_{Q_2}  |F|^p  \,dz\biggr)^2.
\]
Since the first term on the right-hand side can be absorbed into the left-hand side, we apply \eqref{p_poincare_ieq} to get
\[
 \biggl( \fiint_{Q_1} \frac{|u-u_{Q_1}|^{q^*}}{\rho_1^{q^*}}\,dz\biggr)^\frac{2}{q^*}
 \le c\fiint_{Q_2} \frac{|u-u_{Q_2}|^2}{(\rho_2-\rho_1)^2}\,dz
    +c\biggl(\fiint_{Q_2}  |F|^p  \,dz\biggr)^2,
\]
where $c=c(n,N,p,q,q^*,\nu,L,\gamma)$. Employing the interpolation inequality and then using absorbing terms, we get the conclusion. We omit the details.
\end{proof}

Then we discuss a reverse H\"older inequality for the gradient.

\begin{theorem}\label{p_rev_grad}
    Let $\tfrac{2n}{n+2}< p<\infty$.
    Assume that $|F|\in L^p_{\loc}(\Omega_T)$ and that $u\in  L^p_{\loc}(0,T;W^{1,p}_{\loc} (\Om;\mathbb{R}^N))$ is a weak solution to \eqref{p_eq}. 
    Moreover, assume that there exist a parabolic cylinder $2Q=Q_{2\rho}^\la\Subset\Om_T$, an exponent $1<q<p$ satisfying $\max\{2,p\}<\tfrac{q(n+2)}{n+p-q}$ and a constant $\gamma>1$ such that
    \begin{equation}\label{q_stopping}
        \gamma^{-1}\lambda \le \biggl(\fiint_{ \alpha Q } (|\na u|+|F|)^q \, dz\biggr)^\frac{1}{q} \le \lambda
    \end{equation}
     for any $\alpha$ with $1\le \alpha \le 2$.
    Then there exists a constant $c=c(n,N,p,q,\nu,L,\gamma)$ such that
    \[
 \biggl( \fiint_{Q} |\na u|^p \,dz\biggr)^\frac{1}{p}
 \le c \biggl(\fiint_{2Q} |\na u|^q\,dz\biggr)^\frac{1}{q}
    +c\biggl(\fiint_{2Q}  |F|^p  \,dz\biggr)^\frac{1}{p}.
\]
\end{theorem}

\begin{proof}
  Since $\max\{2,p\}<\tfrac{q(n+2)}{n+p-q}$ holds, we apply \eqref{p_poincare_ieq} to the right hand side of \eqref{p_cacc_ie}. Then we get
  \begin{align*}
      \fiint_{Q} |\na u|^p \,dz
       &\le c \biggl(  \fiint_{2Q} |\na u|^q \,dz \biggr)^\frac{p}{q}
       + c \la^{p-2}\biggl(  \fiint_{2Q} |\na u|^q \,dz \biggr)^\frac{2}{q}\\
       &\qquad+c \fiint_{2Q}  |F|^p  \,dz\\
       &\le c \biggl(  \fiint_{2Q} |\na u|^q \,dz \biggr)^\frac{p}{q}\
       +c \fiint_{2Q}  |F|^p  \,dz,
  \end{align*}
  where $c=c(n,N,p,\nu,L,\gamma)$. Here, to obtain the last inequality, we used \eqref{q_stopping}. This completes the proof.
\end{proof}

\begin{remark}
    The condition \eqref{q_stopping} can be achieved for each point $z$ such that $|\na u(z)|+|F(z)|>\la$ via a Calder\'on--Zygmund decomposition. Moreover, based on the reverse H\"older inequality for the gradient in the preceding theorem, the integrability of the gradient has a self-improving property. Specifically there exists $\varepsilon_0=\varepsilon_0(n,N,p,\nu,L)$, $0<\varepsilon_0<1$, such that for every $0<\varepsilon<\varepsilon_0$ and any cylinder $Q=B_\rho\times I_\rho$ with $2Q\Subset \Om_T$, $|F|\in L^{p+\varepsilon}(2Q)$ implies $|\na u|\in L^{p+\varepsilon}(Q)$. The proof can be found in \cite{MR1749438}. 
\end{remark}

Finally, we state a refinement of \cite[Theorem 2.8]{KL_veryweak} for very weak solutions to \eqref{p_weak_sol} in the case of $p$-growth.

\begin{theorem}\label{p_very_weak_thm}
    There exists $\beta_0=\beta_0(n,N,p,\nu,L)\in(0,1)$, $\tfrac{2n}{n+2}<p-\beta_0<p<\infty$, such that if $u\in L_{\loc}^{p-\beta}(0,T;W_{\loc}^{1,p-\beta}(\Om;\mathbb{R}^N))$, $0<\beta<\beta_0$,
    is a very weak solution to \eqref{p_eq} in $\Om_T$ and $|F|\in L^p_{\loc}(\Om_T)$, then $u\in L_{\loc}^p(0,T;W_{\loc}^{1,p}(\Omega;\mathbb{R}^N))$
    and $u$ is a weak solution to \eqref{p_eq} in $\Omega_T$.
\end{theorem}

\begin{proof}
    We apply Theorem~\ref{p_main_left} (i).
    For $\tfrac{2n}{n+2}<p<\infty$ there exists $q$ with $\tfrac{2n}{n+2}<q<p$, sufficiently close to $p$, such that  $2<\tfrac{q(n+2)}{n+p-q}$.
    If \eqref{p_main_left_stopping} holds in $Q=Q_\rho^\la$, then any very weak solution $u\in L_{\loc}^{p-\beta}(0,T;W_{\loc}^{1,p-\beta}(\Om;\mathbb{R}^N))$ for $q<p-\beta$ belongs to $L^2(Q;\mathbb{R}^N)$. As shown in \cite{KL_veryweak}, the $L^2$ integrability assumption is required only within $Q$ and condition \eqref{p_main_left_stopping} is already verified in \cite[(4.2)]{KL_veryweak}. Hence the conclusion follows directly.
    \end{proof}

\section{Time direction estimates for systems with $p$-growth}

Next we discuss estimates in the time direction. For the proof, we will use the following parabolic Poincar\'e inequality for solutions multiplied by the cutoff function.
\begin{lemma}\label{cut_off_p_poincare_lem}
     Let $2< p<\infty$  and $\la>0$. Assume that $|F|\in L^{p-1}_{\loc}(\Omega_T)$ and that $u\in  L^{p-1}_{\loc}(0,T;W^{1,p-1}_{\loc} (\Om;\mathbb{R}^N))$ satisfies \eqref{p_weak_sol}
 for every $\varphi\in C_0^\infty(\Om_T;\mathbb{R}^N)$.
Moreover, assume that $2Q_r\Subset \Om_T$ and that there exist $\eta\in C_0^1(2Q_r)$ and a constant $c>0$ such that
\begin{equation}\label{localize_ftn}
    \| \eta \|_{L^\infty(2Q_r)}+ r^{-1}\| \na \eta \|_{L^\infty(2Q_r)}+ r^{-2}\| \eta_t \|_{L^\infty(2Q_r)}\le c.
\end{equation}
Then there exists a constant $c=c(n,p,L)$ such that
	\begin{equation}\label{cut_off_p_poincare_ie}
        \begin{split}
            &\fiint_{Q}\frac{|u\eta- (u\eta)_Q|}{\rho}\,dz\le c\fiint_Q \biggl( \frac{|u|}{r}+|\na u| \biggr) 1_{2Q_r}\,dz\\
            &\qquad + c\Bigl( \frac{\rho}{r}+1 \Bigr)\la^{2-p} \fiint_Q \biggl( \frac{|u|}{r}+|\na u|+|F|+1\biggr)^{p-1} 1_{2Q_r} \,dz,
        \end{split}
    \end{equation}
    for every parabolic cylinder $Q=\Qla \subset\mathbb{R}^{n+1}$.
\end{lemma}
\begin{proof}
    For any $\varphi\in C_0^\infty(\mathbb{R}^{n+1};\mathbb{R}^N)$ we take $\eta\varphi$ as a test function in \eqref{p_weak_sol}. Then we have
    \begin{align*}
             &\iint_{\Om_T} (- u\eta \cdot \varphi_t +\mathcal{A}(z,\na u)\eta \cdot \na \varphi) \,dz \\
             &\qquad=\iint_{\Om_T}\bigl(u\eta_t\cdot \varphi -\mathcal{A}(z,\na u) \cdot (\na \eta )\varphi +|F|^{p-2}F \eta \cdot \na \varphi\\
             &\qquad\qquad+ |F|^{p-2}F\cdot (\na\eta) \varphi\bigr) \,dz.
    \end{align*}
    We denote $Q=B_\rho\times I^\la_\rho$.
    Let $\varphi\in C_0^1(B)$ satisfy \eqref{gluing_eta}. Then we have
    \begin{align*}
            &\esssup_{t_1,t_2 \in I^\la} |(u\eta)_B(t_1) - (u\eta)_B(t_2) |\le \frac{c}{\rho|B|} \iint_Q (|\na u|^{p-1}+|F|^{p-1})|\eta| \,dz\\
            &\qquad + \frac{c}{|B|}\iint_Q \bigl(|u||\eta_t|+(|\na u|^{p-1}+|F|^{p-1})|\na \eta|\bigr)\,dz,
    \end{align*}
    where $c=c(L)$. Using \eqref{localize_ftn}, we get
    \begin{align*}
            & \rho^{-1}\esssup_{t_1,t_2 \in I^\la} |  (u\eta)_B(t_1) - (u\eta)_B(t_2) |\le c\la^{2-p} \fiint_Q (|\na u|^{p-1}+|F|^{p-1})1_{2Q_r} \,dz\\
            &\qquad + \frac{c\rho}{r}\la^{2-p} \fiint_Q \biggl( \frac{|u|}{r}+|\na u|^{p-1}+|F|^{p-1}\biggr)1_{2Q_r} \,dz.
    \end{align*}
    As in the proof of Lemma~\ref{p_poincare_lem}, we obtain
    \begin{align*}
            &\fiint_{Q}\frac{|u\eta- (u\eta)_Q|}{\rho}\,dz\le c\fiint_Q |\na (u\eta)|\,dz  \\
            &\qquad + c\Bigl( \frac{\rho}{r}+1 \Bigr) \la^{2-p} \fiint_Q \biggl( \frac{|u|}{r}+|\na u|^{p-1}+|F|^{p-1}\biggr)1_{2Q_r} \,dz,
    \end{align*}
    where $c=c(n,L)$. We use \eqref{localize_ftn} and Young's inequality to get
    \begin{align*}
            &\fiint_{Q}\frac{|u\eta- (u\eta)_Q|}{\rho}\,dz\le c\fiint_Q \biggl( \frac{|u|}{r}+|\na u| \biggr) 1_{2Q_r}\,dz\\
            &\qquad + c\Bigl( \frac{\rho}{r}+1 \Bigr)\la^{2-p} \fiint_Q \biggl( \frac{|u|}{r}+|\na u|+|F|+1\biggr)^{p-1} 1_{2Q_r} \,dz,
    \end{align*}
    where $c=c(n,p,L)$. This completes the proof.
\end{proof}

Next, we present pointwise estimates in the time direction. For a cylinder $2Q=2Q_r\Subset\Om_T$, we denote
\begin{equation}\label{p_time_grad}
    g(x,t)=
    \begin{cases}
       \bigl(M^*((|\na u|+|F|) 1_{2Q})(x,t)+1\bigr)^\frac{p}{2},&   1<p<2,\\
         \Bigl( M^*\Bigl(  \Bigl( \tfrac{|u|}{r}  + |\na u|+|F| \Bigr)^{p-1} 1_{2Q}\Bigr)(x,t)  +1\Bigr)^\frac{p}{2(p-1)},  &2<p<\infty.
    \end{cases}
\end{equation}

\begin{lemma}\label{p_point_wise}
    Let $1< p<\infty$ and $q=\max\{1,p-1\}$.
    Assume that $|F|\in L^{q}_{\loc}(\Omega_T)$ and that $u\in  L^q_{\loc}(0,T;W^{1,q}_{\loc} (\Om;\mathbb{R}^N))$ satisfies \eqref{p_weak_sol}
 for every $\varphi\in C_0^\infty(\Om_T;\mathbb{R}^N)$.
Then there exists a constant $c=c(n,p,L)$ such that 
\begin{equation}\label{p_time_der}
    \frac{|u(x,t)-u(x,s)|}{|t-s|^\frac{1}{2}}
    \le c \bigl( g(x,t) + g(x,s) \bigr)
\end{equation}
for every parabolic cylinder $Q=Q_r=B\times I$ such that $2Q\Subset \Om_T$ and almost every $x\in B$ and $t,s\in I$, where $g$ is as in \eqref{p_time_grad}.
\end{lemma}

\begin{proof}
    We divide the proof into two cases.
    
      \textit{Case~1}: $1<p<2$. 
    We set $\varrho=(|t-s|/2)^\frac{1}{2}$, $\rho=\la^\frac{p-2}{2}\varrho$ and
     \[
             \la=
      M^*((|\na u|+|F|) 1_{2Q})(x,t)+M^*((|\na u|+|F|) 1_{2Q})(x,s)+1 .
     \]
     Then we have $\Qla=B_\rho(x)\times I_\rho^\la(\tfrac{t+s}{2}) \subset Q_\varrho$. Note that $I_\rho^\la(\tfrac{t+s}{2})=I_\varrho(\tfrac{t+s}{2})$.
     We apply the scaling invariance property in Lemma~\ref{p_scal_inv} with the cylinder $2\Qla$. Then we follow the proof of Theorem~\ref{time_der} replacing $u$ by $v$ defined in \eqref{p_scale}. For any $x\in B_1$ and $t,s\in I_1$, as in \eqref{average_cal}, \eqref{average_cal_2} and \eqref{average_cal_3}, we have
\begin{align*}
        |v(x,t)-v(x,s)| 
        &\le  |v(x,t) -v_{B_1}(t)|
        + |v(x,s) -v_{B_1}(s)|\\
        &\qquad+ |v_{B_1}(t) -v_{B_1}(s)|\\
        &\le c \bigl(M^*(|\na v|1_{Q_2})(x,t)+M^*(|\na v|1_{Q_2})(x,s)\bigr)\\
        &\qquad+ |(v\varphi)_{B_1}(t) -(v\varphi)_{B_1}(s)|,
\end{align*}
where $c=c(n)$ and $\varphi\in C_0^\infty(B_1)$ satisfies \eqref{gluing_eta}. To estimate the last term on the right-hand side, we use \eqref{p_poincare_cal}. Then we get
\[
|(v\varphi)_{B_1}(t) -(v\varphi)_{B_1}(s)|\le \frac{c}{|B_1|}\iint_{Q_1}   (|\na v|^{p-1}+|G|^{p-1}) \,dz,
\]
where $c=c(L)$.
Since $p-1<1$, by Young's inequality we have
\begin{align*}
         |(v\varphi)_{B_1}(t) -(v\varphi)_{B_1}(s)|
         &\le \frac{c}{|B_1|}\iint_{Q_1}   (|\na v|+|G|+1) \,dz \\
        &\le c( M^*((|\na v| + |G| )1_{Q_2})(x,t) +1),
\end{align*}
where $c=c(p,L)$. Scaling back, we obtain
\begin{align*}
        \frac{|u(x,t)-u(x,s)|}{\la^\frac{p}{2}|t-s|^\frac{1}{2}}
        &=\frac{|u(x,t)-u(x,s)|}{\la\rho}\\
     &\le  c \la^{-1}\bigl( M^*((|\na u| + |F| )1_{2Q})(x,t)\\
     &\qquad + M^*((|\na u| + |F| )1_{2Q})(x,s) +1\bigr).
\end{align*}
     Inserting $\la$, the above inequality becomes
     \begin{align*}
    \frac{|u(x,t)-u(x,s)|}{|t-s|^\frac{1}{2}}
    &\le c \bigl( M^*((|\na u|+|F|) 1_{2Q})(x,t)\\
    &\qquad+M^*((|\na u|+|F|) 1_{2Q})(x,s)+1\bigr)^\frac{p}{2}.
     \end{align*}
     This completes the proof for the case $1<p<2$.

     \textit{Case~2}: $2< p<\infty$. We set 
     $\varrho=(|t-s|/2)^\frac{1}{2}$, $\rho=\la^\frac{p-2}{2}\varrho$ and
     \begin{equation*}
         \begin{split}
        &\la
        =\biggl( M^*\biggl(  \biggl( \frac{|u|}{r}  + |\na u|+|F| \biggr)^{p-1} 1_{2Q}\biggr)(x,t)  \\
      &\qquad + M^*\biggl(  \biggl( \frac{|u|}{r}  + |\na u|+|F| \biggr)^{p-1} 1_{2Q}\biggr)(x,s)+1\biggr)^\frac{1}{p-1} .
         \end{split}
     \end{equation*}
     We denote $Q_0=Q_\rho^\la= B_\rho(x)\times I_{\varrho}(\tfrac{t+s}{2})$ and construct a sequence of intrinsic cylinders $\{Q_i\}_{i=1}^\infty$ such that $(x,t)\in Q_i$, $Q_{i}\subset Q_{i-1}$ and $|Q_{i}|=|\tfrac{1}{2}Q_{i-1}|$ for every $i=1,2,\dots$ with the similarity figure of $Q_0$ and the center points lie in $\{x\}\times \mathbb{R}$. Similarly, we construct $\{Q_j\}_{j=1}^\infty$ only by replacing $(x,s)\in Q_j$ for every $j=1,2,\dots$.
     However, $Q_0\not\subset 2Q$ may occur and we will localize by taking a cutoff function $\eta\in C_0^1(2Q)$ such that $\eta\equiv1$ in $Q$ and that \eqref{localize_ftn} holds.

     We note that
     \begin{align*}
             |u(x,t)-u(x,s)|
             &=|u\eta(x,t)-u\eta(x,s)|\\
             &\le |u\eta(x,t)- (u\eta)_{Q_0}|+|u\eta(x,s)-(u\eta)_{Q_0}|\\
             &\le \sum_{i=1}^{\infty} |(u\eta)_{Q_{i-1}} - (u\eta)_{Q_{i}} | + \sum_{j= 1}^\infty |(u\eta)_{Q_{j-1}} -(u\eta)_{Q_{j}} |,
     \end{align*}
        where
         \begin{align*}
         |(u\eta)_{Q_{i-1}} -(u\eta)_{Q_{i}} |
         &\le \fiint_{Q_{i-1}} | u\eta -  (u\eta)_{Q_{i}} |\,dz\\
         &\le 2^{n+2}\fiint_{Q_{i}} | u\eta -  (u\eta)_{Q_{i}} |\,dz.
         \end{align*}
         Employing \eqref{cut_off_p_poincare_ie}, we get
        \begin{align*}
            &\fiint_{Q_{i}} |u\eta- (u\eta)_{Q_{i}}| \,dz
            \le c 2^{-i}\la^\frac{p-2}{2}\varrho \fiint_{Q_{i}} \biggl( \frac{|u|}{r}+|\na u| \biggr) 1_{2Q}\,dz\\
            &\qquad + c2^{-i}\la^\frac{p-2}{2}\varrho \Bigl( \frac{2^{-i}\la^\frac{p-2}{2}\varrho}{r}+1 \Bigr)\la^{2-p}\\
            &\qquad\qquad\times\fiint_{Q_{i}} \biggl( \frac{|u|}{r}+|\na u|+|F|+1\biggr)^{p-1} 1_{2Q} \,dz,
    \end{align*}
    where $c=c(n,p,L)$.
      Since $(x,t)\in Q_i^\la$, we have
      \[
      \fiint_{Q_{i}} \biggl( \frac{|u|}{r}+|\na u| \biggr) 1_{2Q}\,dz\le M^*\biggl( \biggl( \frac{|u|}{r}+|\na u| \biggr)^{p-1} 1_{2Q} \biggr)^\frac{1}{p-1}  (x,t)\le \la.
      \]
      For the remaining term we consider two cases. If $2^{-i}\la^\frac{p-2}{2}\varrho\le 4r$, then we again have
      \[
      \biggl( \frac{2^{-i}\la^\frac{p-2}{2}\varrho}{r}+1 \biggr)\la^{2-p} \fiint_{Q_{i}} \biggl( \frac{|u|}{r}+|\na u|+|F|\biggr)^{p-1} 1_{2Q} \,dz\le 5\la.
      \]
      On the other hand, if $2^{-i}\la^\frac{p-2}{2}\varrho\ge 4r$, then $ B_{2r}\times I_i \subset Q_i=B_i\times I_i$ and
      \begin{align*}
              &\biggl( \frac{2^{-i}\la^\frac{p-2}{2}\varrho}{r}+1 \biggr)\la^{2-p} \fiint_{Q_{i}} \biggl( \frac{|u|}{r}+|\na u|+|F|\biggr)^{p-1} 1_{2Q} \,dz\\
              &\qquad=\biggl( \frac{2^{-i}\la^\frac{p-2}{2}\varrho}{r}+1 \biggr)\la^{2-p} \left( \frac{2r}{2^{-i}\la^\frac{p-2}{2}\varrho} \right)^n \\
              &\qquad\qquad\times\fiint_{B_{2r}\times I_i }  \biggl( \frac{|u|}{r}+|\na u|+|F|\biggr)^{p-1} 1_{2Q} \,dz\\
              &\qquad\le 2\la^{2-p}\fiint_{B_{2r}\times I_i }  \biggl( \frac{|u|}{r}+|\na u|+|F|\biggr)^{p-1} 1_{2Q} \,dz
              \le 2\la.
      \end{align*}
      Combining the estimates above, it follows that 
      \[
      |(u\eta)_{Q^\la_{i-1}} -(u\eta)_{Q^\la_{i}} |\le c 2^{-i}\la^\frac{p}{2}\varrho.
      \]
      The same calculation works for $j=1,2,\dots$. Therefore, we obtain 
      \[
      |u(x,t)-u(x,s)|\le c\la^\frac{p}{2}\varrho.
      \]
        The proof is completed by inserting $\la$ and $\varrho$.
\end{proof}

\begin{remark}
    We remark that for $1<p<2$, if  $u\in L^p(0,T;W^{1,p}(\Om;\mathbb R^N))$ is a weak solution and there exist a point $t_0\in (0,T)$ and a ball $B\Subset\Om$ such that $|u(\cdot,t_0)|\in L^2(B)$, then \eqref{p_time_der} implies $ |u| \in L^2(B\times (0,T))$. Indeed, adding and subtracting $u(x,t_0)$ gives
    \begin{align*}
            &\iint_{B\times (0,T)} |u(x,t)|^2\,dx\,dt \\
            &\qquad\le c\iint_{B\times (0,T)} \bigl(M^* (( |\na u|+|F| )1_{B\times (0,T)}) (x,t) \bigr)^p \,dx\,dt \\
            &\qquad\qquad + c\int_{B}  \bigl( |u(x,t_0)|^2 + (M^* (( |\na u|+|F| )1_{B\times (t_0-\epsilon,T)}) (x,t_0) )^p \bigr)  \,dx.
    \end{align*}
    The integrability in the spatial direction at a certain time affects the integrability in the space-time direction. This also implies that  weak solutions to \eqref{p_eq} belong to $L^2_{\loc}(\Omega_T;\mathbb R^N)$, if $\tfrac{2n}{n+2}\le p<2$.
\end{remark}

\begin{remark}
    Note that the proof for the case $1<p<2$ can be applicable when $p\ge2$ provided $\Om=\mathbb{R}^n$. In that case \eqref{p_time_der} holds with 
    \[
    g(x,t)=\bigl(M^*((|\na u|+|F|)^{p-1} 1_{2Q})(x,t)+1\bigr)^\frac{p}{2(p-1)}.
    \]
    Moreover, if $Q$ is an intrinsic parabolic cylinder with radius less than $1$  and the upper bound in \eqref{p_main_left_stopping} holds, then the proof of \eqref{p_poincare_ieq} can be adopted and
    \[
    g(x,t)=\bigl(M^*((|\na u|+|F|)^{p-1} 1_{2Q})(x,t)+1\bigr)^\frac{p}{2(p-1)}
    \]
    can be used.
\end{remark}

\begin{lemma}
    Let $2< p<\infty$.
    Assume that $|F|\in L^{p-1}_{\loc}(\Omega_T)$ and that $u\in  L^{p-1}_{\loc}(0,T;W^{1,p-1}_{\loc} (\Om;\mathbb{R}^N))$ satisfies \eqref{p_weak_sol}
 for every $\varphi\in C_0^\infty(\Om_T;\mathbb{R}^N)$. Moreover, assume that there exist a parabolic cylinder $2Q=2\Qla=B_{2\rho}\times I_{2\rho}^\la \Subset\Omega_T$ and a constant $\gamma>1$ such that
\[
   \gamma^{-1}\la\le \biggl(\fiint_{2Q} (|\na u|+|F|)^{p-1}\, dz\biggr)^\frac{1}{p-1}\le \la.
\]
Then there exists a constant $c=c(n,p,L,\gamma)$ such that 
\begin{equation*}
    \begin{split}
        \frac{|u(x,t)-u(x,s)|}{|t-s|^\frac{1}{2}}
    &\le c  ( M^*((|\na u|+|F|)^{p-1} 1_{2Q})(x,t) \\
    &\qquad + M^*((|\na u|+|F|)^{p-1} 1_{2Q})(x,s)+1  )^\frac{p}{2(p-1)}
    \end{split}
\end{equation*}
for almost every $x\in B_\rho$ and $t,s\in I_\rho^\la$.
\end{lemma}

\begin{proof}
    Let $x\in B_\rho$ and $t,s\in I_\rho^\la$. Let $0<\varrho<\rho$ such that 
    \[
    |t-s|=2\la^{2-p}\varrho^2=\left|I_\varrho^\la\left(\tfrac{t+s}{2}\right)\right|.
    \]
    By considering $B_{2\varrho}\times I_{2\varrho}^\la(\tfrac{t+s}{2})$, we apply a scaling argument and follow the argument in the proof of \eqref{p_poincare_ieq}. 
    We denote by $y,\tau,\varsigma$ the points obtained by scaling $x,t,s$, respectively.
    Let
    \begin{align*}
    \La &= 2\cs \bigl(M^*((|\na v|+|G|)^{p-1} 1_{2Q_1})(y,\tau)\\
    &\qquad+ M^*((|\na v|+|G|)^{p-1} 1_{2Q_1})(y,\varsigma)+1\bigr)^\frac{1}{p-1},
    \end{align*}
    where $2\cs$ is defined in \eqref{p_poincare_ieq_lv}.
    We construct a sequence of cylinders $\{ Q_i=Q_{1}^{2^i}(z_i) \}_{i=2}^{ k } $ and $\{Q_i=Q_{\rho_i}^{\La/\rho_i} \}_{i=k+1}^\infty$ such that
\[
 \rho_i=2^{-i+k}, \quad Q_{i+1}\subset Q_i \quad\text{and}\quad \lim_{i\to\infty} v_{Q_i}=v(y,\tau).
\]
Then we have
\begin{align*}
    &|v(y,\tau) - v_{Q_1}|
    \le  c \sum_{i=1}^k \biggl( \biggl(\fiint_{Q_i} g^{p-1}\,dz \biggr)^\frac{1}{p-1}+ 2^{i(2-p)}\fiint_{Q_i} g^{p-1}  \,dz  \biggr) \\
    &\qquad + c \sum_{i=k+1}^\infty \rho_i \biggl( \biggl(\fiint_{Q_i} g^{p-1}\,dz \biggr)^\frac{1}{p-1}+ \La^{2-p}\rho_i^{p-2} \fiint_{Q_i} g^{p-1}  \,dz  \biggr),
\end{align*}
where $c=c(n,p,L)$ and $g=|\na v|+|G|$. We estimate the first sum by $\La$ as in the proof of  \eqref{p_poincare_ieq}, while for the second summation we take the strong maximal operator to present each integral average with $M^*(g^{p-1})(y,\tau)\le \La^{p-1}$. This implies that $|v(y,\tau) - v_{Q_1}|\le c\La$.
The same calculation works for $|v(y,\varsigma) - v_{Q_1}|$. Therefore, by scaling back, we get
\begin{align*}
 \frac{|u(x,t)-u(x,s)|}{\la\varrho}
 &=|v(y,\tau)-v(y,\varsigma)|\\
    &\le \frac{c}{\la} \bigl(M^*((|\na u|+|F|)^{p-1} 1_{2Q_1})(x,t)\\
    &\qquad+ M^*((|\na u|+|F|)^{p-1} 1_{2Q_1})(x,s)+1\bigr)^\frac{1}{p-1}.
\end{align*}
Since $\la^{2-p}\varrho^2=|t-s|$ and for any $z\in 2Q$ we have
\begin{align*}
 \gamma^{-1}
 &\la\le \left(\fiint_{2Q} (|\na u|+|F|)^{p-1}\, dz\right)^\frac{1}{p-1}\\
 &\le M^*( (|\na u|+|F|)^{p-1} 1_{2Q} )^\frac{1}{p-1} (z),
\end{align*}
it follows that
\begin{align*}
    \frac{|u(x,t)-u(x,s)|}{|t-s|^\frac{1}{2}}
    &=\frac{|u(x,t)-u(x,s)|}{\la^\frac{2-p}{2}\varrho} \\
    &\le c\bigl(M^*((|\na u|+|F|)^{p-1} 1_{2Q_1})(x,t)\\
    &\qquad  + M^*((|\na u|+|F|)^{p-1} 1_{2Q_1})(x,s)+1\bigr)^\frac{p}{2(p-1)}.
\end{align*}
This completes the proof.
\end{proof}

Next, we provide a Poincar\'e inequality in the time direction.

\begin{theorem}\label{p_time_poincare}
    Let $1< p <\infty$ and $\max\{1,p-1\}<q<\infty$.
    Assume that $|F|\in L^{\frac{pq}{2}}_{\loc}(\Omega_T)$ and that $u\in  L^\frac{pq}{2}_{\loc}(0,T;W^{1,\frac{pq}{2}}_{\loc} (\Om;\mathbb{R}^N))$ satisfies \eqref{p_weak_sol}
 for every $\varphi\in C_0^\infty(\Om_T;\mathbb{R}^N)$.
Then there exists a constant $c=c(n,p,L)$ such that
\[
    \fint_I |u(x,t)-u_I(x)|^q \,dt\le  c |I|^\frac{q}{2} \fint_I g(x,t)^q \,dt,
\]
for every parabolic cylinder $Q=Q_r=B\times I$ such that $2Q\Subset \Om_T$ and almost every $x\in B$,
where $g$ is as in \eqref{p_time_grad}.
\end{theorem}

\begin{proof}
    We apply Lemma~\ref{p_point_wise} to have
    \begin{align*}
            \fint_I |u(x,t)-u_I(x)|^q \,dt 
            &\le \fint_I\fint_I |u(x,t)-u(x,s)|^q\,ds\,dt\\
            &\le c |I|^\frac{q}{2} \fint_I \fint_I ( g(x,t) + g(x,s)  )^q \,ds \,dt\\
            &= c |I|^\frac{q}{2} \fint_I  g(x,t)^q\,dt.
    \end{align*}
    Therefore, the proof is completed.
\end{proof}

By \cite[Theorem 9.1.15]{HKST2015}, we obtain self-improving properties of the Poincar\'e inequality in the time direction.

\begin{theorem}\label{thm_p_time_der}
Let $1< p<\infty$ and $\max\{1,p-1\}<q<\infty$.
    Assume that $|F|\in L^{\frac{pq}{2}}_{\loc}(\Omega_T)$ and that $u\in  L^\frac{pq}{2}_{\loc}(0,T;W^{1,\frac{pq}{2}}_{\loc} (\Om;\mathbb{R}^N))$ satisfies \eqref{p_weak_sol}
 for every $\varphi\in C_0^\infty(\Om_T;\mathbb{R}^N)$.
    \begin{itemize}
    \item[(i)] If $1< q< 2$, then for $1< q \le q^* <\tfrac{2q}{2-q}$, then there exists a constant $c=c(n,p,q,q^*,L)$ such that 
    \[
    \biggl( \fint_{I } \frac{|u(x,t)-u_{I}(x)|^{q^*}}{|I|^\frac{q^*}{2}}\,dt \biggr)^\frac{1}{q^*}
    \le c\biggl( \fint_{I}  g(x,t)^q\,dt\biggr)^\frac{1}{q}
    \]
    for every parabolic cylinder $Q=Q_r=B\times I$ such that $2Q\Subset \Om_T$ and for almost every $x\in B$,
where $g$ is as in \eqref{p_time_grad}.
    \item[(ii)] If $q=2$, then there exists a constant $c=c(n,p,L)$ such that
    \[
    \fint_I\exp\left(\left(
   \frac{|u(x,t)-u_{I}(x)|}{\displaystyle c|I|^\frac{1}{2}\left( \fint_{I}g(x,t)^{2} \,dt\right)^{\frac{1}{2}}}\right)^{2}\right)\,dt\le c
    \]
    for every parabolic cylinder $Q=Q_r=B\times I$ such that $2Q\Subset \Om_T$ and for almost every $x\in B$,
where $g$ is as in \eqref{p_time_grad}.
    \item[(iii)] If $q>2$, then there exists a constant $c=c(n,p,q,L)$ such that
    \[
    \esssup_{t\in I}\frac{|u(x,t)-u_{I}(x)|}{|I|^\frac{1}{2}}
    \le c\left( \fint_{I} g(x,t)^q\,dt\right)^\frac{1}{q}
    \]
    for every parabolic cylinder $Q=Q_r=B\times I$ such that $2Q\Subset \Om_T$ and for almost every $x\in B$,
where $g$ is as in \eqref{p_time_grad}.
\end{itemize}
\end{theorem}

Let $f\in L^1(\Om_T)$ and $0<h<T$. We define the Steklov average $f_h$ as
\[
f_h(x,t)=
\begin{cases}
    \frac1h\int_t^{t+h} f(x,s)\,ds,&\quad 0<t<T-h,\\
    0,& \quad T-h\le t<0.
\end{cases}
\]
In particular, using integration by parts \eqref{p_weak_sol} can be reformulated as 
\begin{equation}\label{stek_p_eq}
\begin{split}
        &\int_\Om (u_h)_t (x,t) \cdot \varphi(x)+ (\mathcal{A}(\cdot,\na u))_h(x,t)\cdot \na \varphi(x) \,dx \\
    &\qquad=\int_\Om (|F|^{p-2}F)_h(x,t)\cdot \na \varphi(x)\,dx 
\end{split}
\end{equation}
for every $\varphi\in C_0^\infty(\Om;\mathbb{R}^N)$ and almost every $0<t<T-h$.

\begin{theorem}
    Let $1 <p<2$ and $0<h<T$.
    Assume that $|F|\in L^p_{\loc}(\Omega_T)$ and that $u\in  L^p_{\loc}(0,T;W^{1,p}_{\loc} (\Om;\mathbb{R}^N))$ satisfies \eqref{p_weak_sol} for every $\varphi\in C_0^\infty(\Om_T;\mathbb{R}^N)$.
    Then $(u_h)_t \in L_{\loc}^2(\Om_T;\mathbb{R}^N)$ and
    \eqref{stek_p_eq} holds for any $\varphi\in L^2(\Om;\mathbb{R}^N)\cap W_0^{1,p}(\Om;\mathbb{R}^N)$.
    Moreover, if $1<p\le \tfrac{2n}{n+2}$, then $u\in L_{\loc}^\infty(0,T;L_{\loc}^{\frac{np}{n-p} }(\Om;\mathbb{R}^N) )$.
\end{theorem}

\begin{proof}
Note that for almost every $x\in \Om$ and $t\in (0,T-h)$ we have
\[
(u_h)_t(x,t)=\frac{u(x,t+h)-u(x,t)}{h}.
\]
It follows by \eqref{p_time_der} that for any $Q=B_r\times I_r$ such that $2Q\Subset\Om_{T-h}$ we have
\[
|(u_h)_t(x,t)|^2\le c h^{-1}\bigl(g(x,t) + g(x,t+h)\bigr)^2
\]
for almost every $x\in B_r$ and $t\in I_r$, where $g$ is defined in \eqref{p_time_grad}. Since the right-hand side is integrable in $2Q$, we get $|(u_h)_t|\in L^2(Q)$. Moreover, by \eqref{stek_p_eq} we deduce by a density argument that $\varphi\in L^2(\Om;\mathbb{R}^N)\cap W_0^{1,p}(\Om;\mathbb{R}^N)$ is an admissible test function in \eqref{stek_p_eq}.

For $1<p\le \tfrac{2n}{n+2}$ we consider the function
\[
     \frac{u_h(x,t)}{ (1+|u_h(x,t)|^2)^{1-\frac{np}{2(n-p)}} } 
\]
and note that it belongs to $L_{\loc}^2(\Om;\mathbb{R}^N)$ for almost every $t\in(0,T-h)$. Indeed, since $\tfrac{np}{2(n-p)}\le 1$ we have
\[
\biggl( \frac{|u_h|}{ (1+|u_h|^2)^{1-\frac{np}{2(n-p)}} } \biggr)^2 \le (1+|u_h|^2)^{\frac{np}{n-p}-1}\le (1+|u_h|^2)^{\frac{np}{2(n-p)}}
\]
and therefore by Minkowski's integral inequality we get
\begin{align*}
        &\biggl( \fint_B \biggl|   \frac{u_h}{ (1+|u_h|^2)^{1-\frac{np}{2(n-p)}} }  \biggr|^2 \,dx \biggr)^\frac{1}{2} 
        \le  c\biggl( \fint_B |u_h|^\frac{np}{n-p} \,dx \biggr)^\frac{1}{2} +c  \\
        &\qquad\le   c\biggl( \fint_t^{t+h}\biggl( \fint_B  |u|^\frac{np}{n-p}\,dx\biggr)^\frac{n-p}{np} \,dt\biggr)^\frac{np}{2(n-p)} +c  \\
        &\qquad\le c \biggl( \fint_t^{t+h}\biggl( \fint_B  ( |u|+ |\na u| )^p \,dx\biggr)^\frac{1}{p} \,dt \biggr)^\frac{np}{2(n-p)} +c,
\end{align*}
where $c>0$ is a constant independent from $h,t>0$.
Moreover, we observe that
\begin{align*}
\biggl|\na\biggl(  \frac{u_h}{ (1+|u_h|^2)^{1-\frac{np}{2(n-p)}} }  \biggr)  \biggr|
&\le \frac{|\na u_h|}{ (1+|u_h|^2)^{1-\frac{np}{2(n-p)}} }+ \frac{c|u_h|^2 |\na u_h|}{ (1+|u_h|^2)^{2-\frac{np}{2(n-p)}} }\\
&\le c|\na u_h|
\end{align*}
for some $c>0$ and
\[
 \frac{(u_h)_t\cdot u_h}{ (1+|u_h|^2)^{1-\frac{np}{2(n-p)}} } = \frac{n-p}{np} \bigl( (1+|u_h|^2)^{ \frac{np}{2(n-p)}} \bigr)_t .
\]
Let $\varphi\in C_0^\infty(\Om)$ and $0<t_0<T$. For sufficiently small $h,\epsilon>0$ so that $0<t_0<T-h$ we let
\[
\zeta_\epsilon(t)=
\begin{cases}
    1,&\quad 0<t \le t_0-\epsilon,\\
    \frac{-t+t_0}{\epsilon}, &\quad t_0-\epsilon<t\le t_0,\\
    0,&\quad t_0<t<T-h.
\end{cases}
\]
We take
\[
 \frac{u_h\varphi\zeta_\epsilon}{ (1+|u_h|^2)^{1-\frac{np}{2(n-p)}} }\in L^2(\Om;\mathbb{R}^N)\cap W^{1,p}_0(\Om;\mathbb{R}^N)
\]
as a test function in \eqref{stek_p_eq}. Then, applying integration by parts, properties of Steklov averages and Young's inequality, we get
\begin{align*}
        &-\int_\Om (1+|u_h|^2)^{\frac{np}{2(n-p)}}\varphi (\zeta_\epsilon)_t \,dx\\ 
        &\qquad\le c\int_{\Om} \bigl((|\na u_h |+|F_h|)^p\varphi+ (|\na u_h |+|F_h|)^{p-1}(1+|u_h|^2)^{\frac{np}{2(n-p)}-\frac{1}{2} }|\na\varphi|\bigr) \, dx\\
        &\qquad\le c\int_{\Om} \bigl((|\na u_h |+|F_h|)^p\varphi+ (|\na u_h |+|F_h|)^{p-1}(1+|u_h|^2)^{\frac{1}{2} }|\na\varphi|\bigr) \, dx\\
        &\qquad\le c\int_{\Om}  (|u_h|+|\na u_h |+|F_h|+1)^p(\varphi+|\na \varphi|) \, dx,
\end{align*}
where $c=c(n,N,p,L)$.
Taking the integral over $(0,T-h)$ with respect to the time variable and then letting $\epsilon$ and $h$ tend to $0$, we obtain
\[
\int_\Om (1+|u|^2)^{\frac{np}{2(n-p)}}\varphi(x,t_0) \,dx
\le c\iint_{\Om_T} (|u|+|\na u |+|F|+1)^p(\varphi+|\na \varphi|) \, dz.
\]
Since the right-hand side is finite and $0<t_0<T$ is arbitrary, this completes the proof.
\end{proof}

\begin{remark}
    For the case $\tfrac{2n}{n+2}<p<\infty$, since $|u_h(\cdot,t)|\in L^2_{\loc}(\Om)$ for almost every $t\in (0,T-h)$, the standard energy estimate is admissible and therefore the Caccioppoli inequality \eqref{p_cacc_ie} holds without the assumption \eqref{cacc_stopping}.
\end{remark}

\section*{Acknowledgement}
 W. Kim has been supported by the KIAS Individual Grant (HP105501).

\end{document}